\documentclass{amsart}
\usepackage{amsfonts,amssymb,amsmath,amsthm,graphicx,color,amscd,xspace,verbatim,dsfont,mathrsfs}

\usepackage{scalerel}
\makeatletter

\makeatletter
\def\@tocline#1#2#3#4#5#6#7{\relax
  \ifnum #1>\c@tocdepth 
  \else
    \par \addpenalty\@secpenalty\addvspace{#2}%
    \begingroup \hyphenpenalty\@M
    \@ifempty{#4}{%
      \@tempdima\csname r@tocindent\number#1\endcsname\relax
    }{%
      \@tempdima#4\relax
    }%
    \parindent\z@ \leftskip#3\relax \advance\leftskip\@tempdima\relax
    \rightskip\@pnumwidth plus4em \parfillskip-\@pnumwidth
    #5\leavevmode\hskip-\@tempdima
      \ifcase #1
       \or\or \hskip 1em \or \hskip 2em \else \hskip 3em \fi%
      #6\nobreak\relax
      \dotfill
      \hbox to\@pnumwidth{\@tocpagenum{#7}}
    \par
    \nobreak
    \endgroup
  \fi}
\makeatother

\newtheorem{theorem}{Theorem}[section]
\newtheorem{lemma}[theorem]{Lemma}
\newtheorem{proposition}[theorem]{Proposition}

\theoremstyle{definition}
\newtheorem{definition}[theorem]{Definition}

\newtheorem{remark}[theorem]{Remark}

\newcommand{\N}{{\mathbb N}}
\newcommand{\R}{{\mathbb R}}

\newcommand{\tr}{\mathrm{tr}^*}
\newcommand{\loc}{\mathrm{loc}}

\newcommand{\beqn}{\begin{eqnarray}}
\newcommand{\eeqn}{\end{eqnarray}}   
\newcommand{\beq}{\begin{eqnarray*}}
\newcommand{\eeq}{\end{eqnarray*}}

\newcommand{\be}{\small\begin{equation}}
\newcommand{\bel}[1]{\small\begin{equation}\label{#1}}
\newcommand{\ee}{\end{equation}\normalsize}

\newcommand{\BA}{\begin{array}}
\newcommand{\EA}{\end{array}}
\newcommand{\BAN}{\renewcommand{\arraystretch}{1.2}
\setlength{\arraycolsep}{2pt}\begin{array}}
\newcommand{\BAV}[2]{\renewcommand{\arraystretch}{#1}
\setlength{\arraycolsep}{#2}\begin{array}}

\newcommand{\BSA}{\begin{subarray}}
\newcommand{\ESA}{\end{subarray}}
\newcommand{\BAL}{\begin{aligned}}
\newcommand{\EAL}{\end{aligned}}

\newcommand{\forevery}{\quad \forall}

\newcommand{\norm}[1]{\left \|#1\right \|}

\newcommand{\supp}{\mathrm{supp}\,}
\newcommand{\dist}{\mathrm{dist}\,}

\newcommand{\diam}{\mathrm{diam}\,}

\newcommand{\sbs}{\subset}

\def\dist{\mathrm{dist}}

\def \dd {\mathrm{d}}

                         \def\vge{\varepsilon}

\def\gm{\mu}        \def\gn{\nu}         
            \def\vgr{\varrho}
       \def\gt{\tau}
   \def\gv{\vartheta}   \def\gw{\omega}

\def\Gw{\Omega}              

\def\CS{{\mathcal S}}      \def\CN{{\mathcal N}}
   \def\CO{{\mathcal O}}   
\def\CA{{\mathcal A}}   \def\CB{{\mathcal B}}   
      \def\CF{{\mathcal F}}
   \def\CH{{\mathcal H}}

\def\BBA {\mathbb A}   \def\BBB {\mathbb B}    
       
\def\BBG {\mathbb G}   \def\BBH {\mathbb H}    
\def\BBJ {\mathbb J}   \def\BBK {\mathbb K}    
\def\BBM {\mathbb M}   \def\BBN {\mathbb N}    
   \def\BBR {\mathbb R}    
       
\def\BBW {\mathbb W}   \def\BBX {\mathbb X}

\def\GTM {\mathfrak M}

\def\tr{\mathrm{tr}_\mu}

\newcommand{\ei}{\phi_{\xm}}
\newcommand{\xa}{\alpha}
\newcommand{\xb}{\beta}
\newcommand{\xg}{\gamma}
\newcommand{\xG}{\Gamma}
\newcommand{\xd}{\delta}

\newcommand{\xe}{\varepsilon}

\newcommand{\xk}{\kappa}
\newcommand{\xl}{\lambda}
\newcommand{\xL}{\Lambda}
\newcommand{\xm}{\mu}
\newcommand{\xn}{\nu}

\newcommand{\xs}{\sigma}
\newcommand{\xS}{\Sigma}
\newcommand{\xf}{\phi}

\newcommand{\xo}{\omega}
\newcommand{\xO}{\Omega}

\newcommand{\myfrac}[2]{{\displaystyle \frac{#1}{#2} }}

\def \dd {\mathrm{d}}
\def \dx {\mathrm{d}x}

\newcommand{\am}{{\xa_{\scaleto{-}{3pt}}}}

\def\bal#1\eal{\small\begin{align*}#1\end{align*}\normalsize}
\def\ba#1\ea{\small\begin{align}#1\end{align}\normalsize}
\numberwithin{equation}{section}

\begin{document}

\title[Semilinear Schr\"odinger equations with Hardy potentials]{Semilinear Schr\"odinger equations with Hardy potentials involving the distance to a boundary submanifold and gradient source nonlinearities}
\author[K. T. Gkikas]{Konstantinos T. Gkikas}
\address{K.T. Gkikas, Department of Mathematics, University of the Aegean, 
83200 Karlovassi, Samos, Greece\newline
Department of Mathematics, National and Kapodistrian University of Athens, 15784 Athens, Greece
}
\email{kgkikas@aegean.gr}
\author[M. Paschalis]{Miltiadis Paschalis}
\address{M. Paschalis, Department of Mathematics, National and Kapodistrian University of Athens, 15784 Athens, Greece}
\email{mpaschal@math.uoa.gr}
\date{\today}

\begin{abstract} 
Let $\Omega\subset\R^N$ ($N\geq 3$) be a bounded $C^2$ domain and $\Sigma\subset\partial\Omega$ be a compact $C^2$ submanifold of dimension $k$. Denote the distance from $\Sigma$ by $d_\Sigma$. In this paper, we study positive solutions of the equation $(*)\, -\Delta u -\mu u/d_\Sigma^2 = g(u,|\nabla u|)$ in $\Omega$, where $\mu\leq \big( \frac{N-k}{2} \big)^2$ and the source term $g:\R\times\R_+ \rightarrow \R_+$ is continuous and non-decreasing in its arguments with $g(0,0)=0$. In particular, we prove the existence of solutions of $(*)$ with boundary measure data $u=\nu$ in two main cases, provided that the total mass of $\nu$ is small. In the first case $g$ satisfies some subcriticality conditions that always ensure the existence of solutions. In the second case we examine power type nonlinearity $g(u,|\nabla u|) = |u|^p|\nabla u|^q$, where the problem may not possess a solution for exponents in the supercritical range. Nevertheless we obtain criteria for existence under the assumption that $\nu$ is absolutely continuous with respect to some appropriate capacity or the Bessel capacity of $\Sigma$, or under other equivalent conditions. 
	
\medskip


\medskip

\end{abstract}

\maketitle
\tableofcontents

\section{Introduction}

Let $\Omega\subset\R^N$ ($N\geq 3$) be a bounded domain of class $C^2$, $\Sigma\subset\partial\Omega$ a compact $C^2$ submanifold of the boundary of dimension $k\in\{0,\ldots,N-1 \}$. Let $\mu\leq \big(\frac{N-k}{2}\big)^2$ and denote $d_\Sigma(x):=\mathrm{dist}(x,\Sigma)$. In this paper, we investigate the existence of solutions to the boundary value problem with measure data for the semilinear elliptic equation
\bal
-L_\mu u = g(u,|\nabla u|)\quad \text{in }\Omega,
\eal
where $L_\mu:= \Delta + \mu d_\Sigma^{-2}$ and $g:\R\times\R_+ \rightarrow \R_+$ is a continuous and non-decreasing in its two arguments nonlinear source term with $g(0,0)=0$. The term $\mu d_\Sigma^{-2}$ is called a Hardy potential in view of the related Hardy inequalities in which such terms appear and which are a central feature of the analysis of Schr\"odinger operators with singular potentials. A particularly important type of nonlinearity to keep in mind is the power-type $g(u,|\nabla u|) = |u|^p|\nabla u|^q$, with which a large portion of this paper is concerned.

\subsection{Background and main results}

The study of Schr\"odinger operators has long been a central feature of the theory of elliptic and parabolic partial differential equations. In recent years there has been a surge of research regarding such operators involving singular potentials, particularly Hardy-type potentials, as well as the parallel study of related Hardy inequalities which are a key aspect of that theory; for example see \cite{BEL,BFT,BFT3,BM,BMS,C,DD2,FF,FT,GM,hardy-marcus} and references therein. 

An important class of differential equations arising in this setting is without doubt the semilinear Schr\"odinger equations. If the related Schr\"odinger operator is $-\Delta$, without any additional potential, the equation usually assumes the general form $-\Delta u = \pm g(u,|\nabla u|)$, subject to boundary conditions given in terms of some appropriate boundary trace. A theory for this is already well-established and now considered somewhat classical, see e.g. \cite{BVi,GV,MVbook,Vbook}; note however that there is still activity in the area as the framework is very general and admits a number of improvements, see e.g. \cite{BHV,BVN}.

In \cite{BHV} Bidaut-V\'eron, Hoang and V\'eron investigate existence of solutions of the boundary value problem with measure data for equation with mixed source term $-\Delta u = |u|^{q_1}|\nabla u|^{q_2},$ where $q_1,q_2\geq 0$ are such that $q_1+q_2>1$ and $q_2<2$. In particular, they provide equivalent conditions for existence provided that the mass of the boundary measure is small by extensively utilising the abstract setting of integral equations developed by Kalton and Verbitsky \cite{KV}.

In recent years there have been developments of the theory in the case where a Hardy-type potential is present; in this case the relevant Schr\"odinger operator is $-\Delta-\mu/d_\Sigma^2$ where $\Sigma$ is a submanifold of $\overline{\Omega}$, $d_\Sigma= \mathrm{dist}(\cdot,\Sigma)$ and $\mu$ is a constant. This case features the added difficulty of the presence of a strong singularity in the domain or its boundary. For the relevant literature, see e.g. \cite{BMM,BMR,D,DN,GkiNg_absorption,GN2,GkiNg_linear,GN-gradient,GN1,GkV,MarMor,MT,MarNgu}. In recent paper \cite{GN-gradient}, Gkikas and Nguyen investigate solutions of equation $-\Delta u -\mu u/ d_{\partial\Omega}^{2} = g(u,|\nabla u|)$ with measure boundary data, among other things, and prove the existence of solutions in both a subcritical and supercritical setting, using Green and Martin kernel estimates established by Fillipas, Moschini and Tertikas \cite{FMoT2}.

In the present, we use the recently derived Green and Martin kernel estimates established by Barbatis, Gkikas and Tertikas \cite{BGT} to study semilinear Schr\"odinger equations with gradient-dependent source terms, in which the Hardy potential has the singularity on some submanifold of the boundary. Let $N\geq 3,$ $\xO\subset\mathbb{R}^N$ be a $C^2$ bounded domain and $\xS \subset \partial\Omega$ be a $C^2$ compact submanifold in $\mathbb{R}^N$ without boundary, of dimension $k$, where $0\leq k \leq N-1$. We assume that  $\xS = \{0\}$ if  $k = 0$ and $\xS=\partial\xO$ if $k=N-1.$ Let $d(x)=\dist(x,\partial \Omega)$ and $d_\xS(x)=\dist(x,\xS)$. Let $\xm\leq \big(\frac{N-k}{2}\big)^2$ be a parameter and
\bal
L_\mu u: =\Delta u + \frac{\mu}{d_\xS^2}u\quad \textrm{in }\Omega.
\eal

In the case $\xS\subsetneqq\partial\xO$ where $0 \leq k\leq N-2$, the linear equation $L_\xm u=0$ was extensively investigated in recent papers \cite{BGT,MT} where the optimal Hardy constant
\bal
C_{\xO,\xS}:=\inf_{u\in H_0^1(\xO)}\frac{\int_\xO|\nabla u|^2dx}{\int_\xO |u|^2 d_\xS^{-2}dx}
\eal
is deeply involved. It is known (see e.g. \cite{FF}) that  $0<C_{\xO,\xS} \leq \big(\frac{N-k}{2}\big)^2$. Moreover, when  $\xm< \big(\frac{N-k}{2}\big)^2$, there exists a minimizer $\ei\in H_0^1(\xO)$ of the Dirichlet eigenvalue problem
\ba\label{eigenvalue}
\xl_{\xm}:=\inf_{u\in H_0^1(\xO)}\frac{\int_\xO|\nabla u|^2dx-\xm\int_\xO |u|^2 d_\Sigma^{-2}dx}{\int_\xO |u|^2 dx}>-\infty.
\ea
If $\xm=\big(\frac{N-k}{2}\big)^2$ then \eqref{eigenvalue} holds true, however there is no minimizer in $H_0^1(\xO)$. The reader is referred to \cite{FF} for more detail. In addition, by \cite[Proposition A.2]{BGT} (see also \cite[Lemma 2.2]{MT}), the corresponding eigenfunction $\ei$ satisfies the following pointwise estimate
\bal
\ei(x) \approx d(x)d^{-\am}_\Sigma(x),\quad x\in\Omega ,
\eal
where
\bal
\xa_\pm:=H\pm\sqrt{H^2-\xm} \quad\text{and}\quad H:=\frac{N-k}{2}.
\eal
Our main assumption in this paper is that $\lambda_\mu>0$ and $\mu\leq H^2$; this is always the case when $\mu<C_{\Omega,\Sigma}$, but not exclusively so.

Our aim is to study the existence of weak solutions for the boundary value problem
\be\label{intro-bvp}
\left\{ \begin{array}{ll}
	-L_\mu u &= g(u,|\nabla u|) \text{ in }\Omega\\
	\tr(u)&=\vgr\nu\\
\end{array}
\right.
\ee
where the nonlinear source term $g:\R\times\R_+ \rightarrow \R$ is continuous and non-decreasing in its arguments, $\nu$ is a Radon measure on $\partial\Omega$ such that $\| \nu \|_{\GTM(\partial\Omega)}=1$,  $\vgr$ is a positive parameter and $\tr u$ is an appropriate trace which is defined in a dynamic way in \cite{BGT,GkV} as follows.

\begin{definition}
Let $p>1$. A function $u\in W^{1,p}$ possesses an $L_\mu$-boundary trace $\nu\in\GTM(\partial\Omega)$ (with basis at $x_0\in\Omega$) if for any smooth exhaustion $\{\Omega_n\}$ of $\Omega$ there holds
\bal
\lim_{n\rightarrow\infty} \int_{\partial\Omega_n} \phi u\,d\omega^{x_0}_{\Omega_n} = \int_\Omega \phi\,d\nu \quad \forall \phi\in C(\overline{\Omega}),
\eal
where $\omega^{x_0}_{\Omega_n}$ is the $L_\mu$-harmonic measure (with basis at $x_0$).
\end{definition}

\noindent
Note that this is equivalent to $u d\omega^{x_0}_{\Omega_n} \rightharpoonup d\nu$ in $\GTM(\overline{\Omega})$. If this is the case we write $\tr u := \nu$. A weak solution of \eqref{intro-bvp} is now defined as follows.

\begin{definition}
We say that $u$ is a \textit{weak solution} of \eqref{intro-bvp} if $u\in L^1(\Omega;\ei)$, $g(u,|\nabla u|)\in L^1(\Omega;\ei)$ and 
\bal
-\int_\Omega uL_\mu \zeta\,dx = \int_\Omega g(u,|\nabla u|)\zeta\,dx - \vgr \int_\Omega \BBK_\mu[\nu]L_\mu\zeta\,dx \quad \forall \zeta \in \BBX_\mu(\Omega,\Sigma).
\eal
where
\bal
\BBX_\mu(\Omega,\Sigma):= \{ \zeta\in H^1_\loc(\Omega): \ei^{-1}\zeta\in H^1(\Omega;\ei^2) \text{ and } \ei^{-1}L_\mu\zeta \in L^\infty(\Omega) \}
\eal
and $\BBK_\mu$ is the Martin operator related to $L_\mu$.
\end{definition}

\noindent
This will be clarified in the next section.

The layout of the paper is as follows. In Section \ref{preliminaries} we provide the necessary background material, as well as some technical tools to be used in the sequel. In Section \ref{weak-lebesgue}, we derive weak Lebesgue estimates for the Green and Martin operators related to $L_\mu$. In fact, our results are general enough to include estimates for the gradient of these operators, a feature necessary for the treatment of gradient-dependent nonlinearities.

In Section \ref{subctitical-problem}, we prove the existence of solutions for \eqref{intro-bvp} under some subcriticality conditions, in particular when $g$ is such that
\bal
\Lambda_g:= \int_1^\infty g(s,s^{p_*/q_*})s^{-1-p_*}\,ds < \infty \qquad \text{and} \qquad g(as,bt)\leq c(a^p+b^q)g(s,t)
\eal
for some $p,q>1$, $c>0$ and all $a,b,s,t\in \R_+$ ($p_*$ and $q_*$ are critical exponents, to be determined later). The main tool used here, aside from other technicalities developed in order to ensure its applicability, is the Schauder fixed point theorem, which has been successfully used in the treatment of similar problems; see for example \cite{BHV,GN1,GN-gradient,GkiNg_source}. The proof also depends on an auxiliary result where $g$ is assumed to be sufficiently smooth and bounded, as well as the Vitali convergence theorem. Our main result is the following.

\begin{theorem}\label{subcritical}
Let $\nu\in \GTM_+(\partial\Omega)$ with $\|\nu\|_{\GTM(\partial\Omega)}=1$ and \bal
p_*:= \min \bigg\{ \frac{N+1}{N-1},\frac{N-\xa_-+1}{N-\xa_--1} \bigg\},\quad q_*:= \min \bigg\{ \frac{N+1}{N},\frac{N-\xa_-+1}{N-\xa_-} \bigg\}.
\eal
Assume that  $g$ satisfies 
\ba\label{subcond}
\Lambda_g:= \int_1^\infty g(s,s^{p_*/q_*})s^{-1-p_*}\,ds < \infty \qquad \text{and} \qquad g(as,bt)\leq c(a^p+b^q)g(s,t)
\ea
for some $p,q>1$, $c>0$ and all $a,b,s,t\in \R_+$. Then there exists a positive $\vgr_0=\vgr(\mu,\Omega,\xL_g,c,p,q)$ such that the problem
\be\label{subcrit-bvp}
\left\{ \begin{array}{ll}
	-L_\mu u &= g(u,|\nabla u|) \text{ in }\Omega\\
	\tr(u)&=\vgr\nu\\
\end{array}
\right.
\ee
possesses a positive weak solution for all $\vgr\in (0,\vgr_0)$.
\end{theorem}

Finally, in Section \ref{supercritical-source}, we focus on the specific case $g(u,|\nabla u|) = |u|^p |\nabla u|^q$, and rely on alternative methods to extend the results of Section \ref{subctitical-problem} in a range of exponents beyond the subcritical. Although the Schauder fixed point theorem is still the key component of the proofs, we mostly rely on potential-theoretic methods to demonstrate its applicability, in particular we specialise and extensively utilise the abstract setting of integral equations developed by Kalton and Verbitsky \cite{KV}. In these results the existence of a solution is demonstrated mostly under the assumption that the boundary measure $\nu$ is absolutely continuous with respect to some appropriate capacity (involving either a suitable kernel $\CN_{\xa,\xs}$ defined in Section \ref{supercritical-source} or the Bessel kernel), or under other equivalent conditions.

In this direction, in the subcritical case, we show that the critical exponents depend on concentration of $\xn.$ 

\begin{theorem}\label{powersubcritical}
Let $\mu<H^2$, $p,q\geq 0$ such that $p+q>1$ and $(p+q-1)\xa_-<p+1$, and let $\nu\in\GTM_+(\partial\Omega)$ be such that $\| \nu \|_{\GTM(\partial\Omega)}=1$. Assume that either of the following conditions holds:
\begin{enumerate}
\item $\supp \nu \subset\Sigma$ and $(N-\xa_-)(p+q-1)<p-1$,
\item $\supp \nu \subset\partial\Omega\setminus\Sigma$ and $N(p+q-1)<p-1$.
\end{enumerate}
The the boundary value problem 
\be\label{intro-power-BVP} \left\{ \BAL
- L_\gm u&= |u|^p|\nabla u|^q &\text{ in } \Omega\\
\tr(u)&=\vgr \nu &
\EAL \right.
\ee 
admits a non-negative weak solution provided that $\vgr$ is small enough.
\end{theorem}
\noindent Note that the function $g(t,s) = t^p s^q$ satisfies \eqref{subcond} with $p_*=\frac{N-\am+1}{N-\am-1}$ and $q_*=\frac{N-\am+1}{N-\am}$ if $(N-\xa_-)(p+q-1)<p-1$ and with  $p_*=\frac{N+1}{N-1}$ and $q_*=\frac{N+1}{N}$ if $N(p+q-1)<p-1$.

In the supercritical case, we provide sufficient conditions for the existence of the solutions in terms of an appropriate Bessel capacity defined in  \eqref{besselcapacity}.

\begin{theorem}\label{intro-Bessel-Sigma}
Let $p,q\geq 0$ with $p+q>1$ and $$\xa_-<\frac{p+1}{p+q-1}$$ and let $\nu\in \GTM_+(\partial \Omega)$ with $\supp \nu \subset \Sigma$ be such that
\bal
\nu(B) \leq C \mathrm{Cap}^{\Sigma,(p+q)'}_{\gv}(B) \quad \forall B\in \CB(\Sigma),
\eal
where
\bal
\gv:= k-N+2\xa_-+\frac{N-k+p+1-(p+q+1)\xa_-}{p+q},
\eal
and assume that $0<\gv<k$ and $N-k>(p+q+1)\xa_--p-1$. Then \eqref{intro-power-BVP} possesses a non-negative weak solution provided that $\vgr\in (0,\vgr_0)$ for small enough $\vgr_0$.
\end{theorem}

\begin{theorem}\label{intro-Bessel-Omega}
Let $p,q\geq 0$ with $p+q>1$ and $$\xa_-<\frac{p+1}{p+q-1}$$ and let $\nu\in \GTM_+(\partial \Omega)$ with $\supp \nu \subset \Omega\setminus\Sigma$ be such that
\bal
\nu(B) \leq C \mathrm{Cap}^{\partial\Omega,(p+q)'}_{\frac{2-q}{p+q}}(B) \quad \forall B\in \CB(\partial\Omega),
\eal
where we assume that $q<2$ and $N(p+q)>p+2$. Then \eqref{intro-power-BVP} possesses a non-negative weak solution provided that $\vgr\in (0,\vgr_0)$ for small enough $\vgr_0$.
\end{theorem}
\noindent In the case $\Sigma=\{0\}$ and $\mu=H^2$ the above result is valid provided $p,q\geq 0$ with $p+q>1$ and $N+2-(N-2)p-Nq-2\xe>0$ for some $\xe>0$ (see Theorem \ref{0-boundary-existence}).

Finally, we examine the special case $\Sigma=\{0 \}$ and $\mu=H^2$.

\begin{theorem}\label{intro-0-bvp}
Suppose that $\Sigma=\{ 0 \}$ and $\mu=H^2$. Let $p,q\geq 0$ with $p+q>1$ and let $\nu\in\GTM_+(\partial\Omega)$ be such that $\| \nu \|_{\GTM(\partial\Omega)}=1$, and assume that $$N+2-(N-2)p-Nq-2\xe>0.$$ Moreover, assume that either of the following conditions hold:
\begin{enumerate}
\item $\nu=\delta_0$ (the Dirac measure concentrated at zero) and $N+p+1-\frac{N}{2}(p+q+1)-\xe(p+q)>0$.
\item $\supp \nu \subset \partial \Omega \setminus \{ 0\}$ and $N(p+q-1)<p+1$.
\end{enumerate}
Then the boundary value problem \eqref{intro-power-BVP} possesses a non-negative weak solution provided that $\vgr$ is small enough.
\end{theorem}

Our results generalize those regarding the source case in Gkikas and Nguyen \cite{GN-gradient}, where the case $\Sigma=\partial \Omega$ is treated; see also the pioneering work of Bidaut-V\'eron, Hoang, V\'eron \cite{BHV} for the case $\mu=0$ (without Hardy potential). To our knowledge, this is the first instance in which critical boundary singularities in semilinear elliptic problems are treated in such generality.

\medskip
\noindent \textbf{Acknowledgement.} The research project was supported by the Hellenic Foundation for Research and Innovation
(H.F.R.I.) under the “2nd Call for H.F.R.I. Research Projects to support Post-Doctoral Researchers” (Project
Number: 59).

\section{Preliminaries}\label{preliminaries}

\subsection{The submanifold $\Sigma\subset\partial\Omega$}

We are now going to introduce some notation and tools that will be useful for our local analysis near $\Sigma$ and $\partial\xO$; see e.g. \cite{Kuf,BGT}.

Let $x =(x',x'')\in \R^N$, $x'=(x_1,..,x_{N-k}) \in \R^{N-k}$,  $x''=(x_{N-k+1},...,x_N) \in \R^{k}$. For $\beta>0$, we denote by $B^{N-k}(x',\beta)$ the ball  in $\R^{N-k}$ with center $x'$ and radius $\beta$.
For any $\xi\in \Sigma$ we also set
\[
V_\Sigma(\xi,\xb):=
\Big\{x=(x',x''): |x''-\xi''|<\beta,\; |x_i-\Gamma_{i,\Sigma}^\xi(x'')|<\xb,\;\forall i=1,...,N-k\Big\},
\]
for some functions $\Gamma_{i,\Sigma}^\xi: \R^{k} \to \R$, $i=1,...,N-k$.

Let $\Sigma_\beta := \{ x\in \Omega: d_\Sigma(x)<\beta \}$. Since $\Sigma$ is a $C^2$ compact submanifold in $\mathbb{R}^N$ without boundary,
there exists $\xb_0>0$ such that

\begin{itemize}
\item For any $x\in \Sigma_{6\beta_0}$, there is a unique $\xi \in \Sigma$  satisfying
$|x-\xi|=d_\Sigma(x)$.

\item $d_\Sigma \in C^2(\Sigma_{4\beta_0})$, $|\nabla d_\Sigma|=1$ in $\Sigma_{4\beta_0}$ and there exists $g\in L^\infty(\Sigma_{4\beta_0})$ such that
\[
\Delta d_\Sigma(x)=\frac{k-1}{d_\Sigma(x)}+g(x) , \quad \text{in } \Sigma_{4\beta_0} .
\]
(See \cite[Lemma 2.2]{Vbook} and \cite[Lemma 6.2]{DN}.)

\item For any $\xi \in \Sigma$, there exist $C^2$ functions $\Gamma_{i,\Sigma}^\xi \in C^2(\R^{k};\R)$, $i=1,...,N-k$, such that defining
\[
V_\Sigma(\xi,\xb):=
\Big\{x=(x',x''): |x''-\xi''|<\beta,\; |x_i-\Gamma_{i,\Sigma}^\xi(x'')|<\xb,\;
 i=1,...,N-k\Big\},
\]
we have (upon relabelling and reorienting the coordinate axes if necessary)
\[
V_\Sigma(\xi,\beta) \cap \Sigma=
\Big\{x=(x',x''): |x''-\xi''|<\beta,\;
x_i=\Gamma_{i,\Sigma}^\xi (x''), \;  i=1,...,N-k\Big\}.
\]

\item There exist $\xi^{j}$, $j=1,...,m_0$, ($ m_0 \in \N$) and $\beta_1 \in (0, \beta_0)$ such that
\be \label{cover}
\Sigma_{2\xb_1}\subset \bigcup_{i=1}^{m_0} V_\Sigma(\xi^i,\beta_0).
\ee
\end{itemize}

Now set
\[
\xd_\Sigma^\xi(x):=\Big(\sum_{i=1}^{N-k}|x_i-\Gamma_{i,\Sigma}^\xi(x'')|^2\Big)^{\frac{1}{2}}, \quad x=(x',x'')\in V_K(\xi,4\beta_0).
\]
Then there exists a constant $C=C(N,K)$ such that
\be\label{propdist}
d_\Sigma(x)\leq	\xd_\Sigma^{\xi}(x)\leq C \| \Sigma \|_{C^2} d_\Sigma(x),\quad \forall x\in V_\Sigma(\xi,2\beta_0),
\ee
where $\xi^j=((\xi^j)', (\xi^j)'') \in \Sigma$, $j=1,...,m_0$, are the points in \eqref{cover} and
\[
\| \Sigma \|_{C^2}:=\sup\{  \| \Gamma_{i,\Sigma}^{\xi^j} \|_{C^2(B_{5\beta_0}^{k}((\xi^j)''))}: \; i=1,...,N-k, \;j=1,...,m_0 \} < \infty.
\]
For simplicity we shall write $\xd_\Sigma$ instead of $\xd_\Sigma^{\xi}$.
Moreover, $\beta_1$ can be chosen small enough so that for any $x \in \Sigma_{\beta_1}$,
\[
B(x,\beta_1) \subset V_\Sigma(\xi,\beta_0),
\]
where $\xi \in \Sigma$ satisfies $|x-\xi|=d_\Sigma(x)$.

When $\Sigma=\partial\xO$ we assume that

\[
V_{\partial\xO}(\xi,\beta) \cap \xO=\Big\{x: \sum_{i=2}^N|x_i-\xi_i|^2<\beta^2,\;
0< x_1 -\Gamma_{1,\partial\xO}^\xi (x_2,...,x_N) <\beta \Big\}.
\]
Thus, when $x\in \Sigma\subset\partial\xO$ is a $C^2$ compact submanifold in $\mathbb{R}^N$ without boundary, of dimension $k$, $0\leq k \leq N-1,$ we have that

\be
\xG_{1,\Sigma}^\xi(x'')=\xG_{1,\partial\xO}^\xi(\xG_{2,\Sigma}^\xi(x''),...,\xG_{N-k,\Sigma}^\xi(x''),x'').\label{dist3}
\ee

Let $\xi\in \Sigma$. For any $x\in V_\Sigma(\xi,\xb_0)\cap \xO,$ we define
\[
\xd(x)=x_1-\Gamma_{1,\partial\xO}^\xi (x_2,...,x_N),
\]
and
$$\xd_{2,\Sigma}(x)=\Big(\sum_{i=2}^{N-k}|x_i-\Gamma_{i,\Sigma}^\xi(x'')|^2 \Big)^{\frac{1}{2}} .
$$
Then by \eqref{dist3}, there exists a constant $A>1$ which depends only on
$\xO $, $\Sigma$ and $\beta_0$ such that

\be
\frac{1}{A}(\xd_{2,\Sigma}(x)+\xd(x))\leq\xd_\Sigma(x)\leq A(\xd_{2,\Sigma}(x)+\xd(x)), \label{propdist2}
\ee
hence by \eqref{propdist} and \eqref{propdist2} there exists a constant $C=C(\xO,\Sigma,\xg)>1$ which depends on $k, N,\Gamma_{i,\Sigma}^\xi,\Gamma_{1,\partial\xO}^\xi,\xg $ such that

\bal
C^{-1}\xd^2(x)(\xd_{2,\Sigma}(x)+\xd(x))^\xg\leq d^2(x)d_\Sigma^\xg(x)\leq C\xd^2(x)(\xd_{2,\Sigma}(x)+\xd(x))^\xg.
\eal

\medskip

With these set, we recall the following estimate which is used several times in the paper.

\begin{lemma}[{\cite[Lemma A.1]{GkiNg_source}}] \label{lemapp:1}
Assume $\ell_1>0$, $\ell_2>0$, $\alpha_1$ and $\alpha_2$ such that $N-k+\alpha_1 + k\alpha_2 >0$. For $y \in \Omega$, put
$$
\CA(y):= \{ x \in \Omega : d_{\Sigma}(x) \leq \ell_1 \text{ and } |x-y| \leq \ell_2 d_{\Sigma}(x)^{\alpha_2} \}.
$$
Then
\bal
\int_{\CA(y) \cap \Sigma_{\beta_1}} d_{\Sigma}(x)^{\alpha_1}\dx \lesssim \ell_1^{N-k+\alpha_1 + k\alpha_2}\ell_2^k.
\eal
\end{lemma}

\subsection{First eigenvalue of $-L_\mu$}

Here we review some basic properties of the first eigenvalue and corresponding eigenfunction of $-L_\mu$; for more details, see \cite{BGT}. For $\mu\leq \big( \frac{N-k}{2} \big)^2$, it is known that
\bal
\xl_\xm:=\inf_{u\in H_0^1(\xO)\setminus \{0\}}\frac{\int_\xO|\nabla u|^2dx-\xm\int_\xO\frac{u^2}{d_\Sigma^2}dx}{\int_\xO u^2 dx} > -\infty.
\eal
We set
\bal
H:=\frac{N-k}{2},\quad \xa_\pm := H \pm \sqrt{H^2-\mu}.
\eal
If  $\xm<H^2$, then there exists a minimizer $\ei\in H_0^1(\xO)$ of \eqref{eigenvalue};
see \cite{FF} for more details. In addition, by \cite[Lemma 2.2]{MT} the eigenfunction $\ei$ satisfies
\bal
\xf_\xm(x) \approx d(x)d_\Sigma^{-\xa_-}(x)\quad \textrm{for } x\in\Omega,
\eal
provided that $\xm<C_{\xO,K}$. If, on the other hand, $\xm=H^2$ then there is no $H_0^1(\xO)$ minimizer.  However, there exists a function $\xf_{\mu}\in H^1_{loc}(\xO)$ such that $-L_{\mu} \xf_{\mu}=\xl_{\mu}\xf_{\mu}$ in $\xO$ in the sense of distributions; see \cite[Proposition A.2]{BGT} for a proof of this.

\subsection{Two-sided estimates on Green function and Martin kernel} 

In this subsection, we recall sharp two-sided estimates on the Green function $G_\xm$ and the Martin kernel $K_\xm$ associated to $-L_\mu$ in $\Omega$. Estimates on the Green function and the Martin kernel are stated in the following Propositions.

\begin{proposition}[{\cite[Proposition 5.3]{BGT}}]  Assume that $\xm \leq \frac{(N-k)^2}{4}$ and $\lambda_{\mu}>0$.
	
$(i)$ If $\am<\frac{N-k}{2}$ or $\am=\frac{N-k}{2}$ and $k\neq 0,$ then for any $x,y\in\xO$ and $x\neq y,$ there holds
\bal
	G_{\xm}(x,y)&\approx
	\min\left\{\frac{1}{|x-y|^{N-2}},\frac{d(x)d(y)}{|x-y|^N}\right\} \left(\frac{\left(d_\xS(x)+|x-y|\right)\left(d_\xS(y)+|x-y|\right)}{d_\xS(x)d_\xS(y)}\right)^{\am}.
\eal

$(ii)$ If $\am=\frac{N}{2}$ and $k=0,$ then for any $x,y\in\xO$ and $x\neq y,$ there holds
\ba \label{Greenestb}\BAL
	G_{\xm}(x,y)&\approx
\min\left\{\frac{1}{|x-y|^{N-2}},\frac{d(x)d(y)}{|x-y|^N}\right\}\left(\frac{\left(|x|+|x-y|\right)\left(|y|+|x-y|\right)}{|x||y|}\right)^{-\frac{N}{2}}\\
	&\hspace{1cm}  +\frac{d(x)d(y)}{(|x||y|)^{\frac{N}{2}}}\left|\ln\left(\min\left\{|x-y|^{-2},(d(x)d(y))^{-1}\right\}\right)\right|.
\EAL
\ea
\end{proposition}

\begin{proposition}[{\cite[Theorem 2.8]{BGT}}] ~~
Assume that $\xm \leq \frac{(N-k)^2}{4}$ and $\lambda_{\mu}>0$.

$(i)$ If $\xm <\frac{(N-k)^2}{4}$ or $\xm=\frac{(N-k)^2}{4}$ and $k>0$ then
\be
\label{Martinest1}
K_{\xm}(x,\xi) \approx\frac{d(x)}{|x-\xi|^N}\left(\frac{\left(d_\xS(x)+|x-\xi|\right)^2}{d_\xS(x)}\right)^{\am} \quad \text{for all } x \in \Omega, \, \xi \in \partial\xO.
\ee

$(ii)$ If  $\xm=\big(\frac{N}{2}\big)^2$ and $k=0$ then
\bal
K_{\mu}(x,\xi) \approx \frac{d(x)}{|x-\xi|^N}\left(\frac{\left(|x|+|x-\xi|\right)^2}{|x|}\right)^{\frac{N}{2}}
+\frac{d(x)}{|x|^{\frac{N}{2}}}\left|\ln\left(|x-\xi|\right)\right|,\quad
\text{ for all } x \in \Omega, \, \xi \in \partial\xO.
\eal
\end{proposition}

Recall that the Green operator and Martin operator are respectively defined by
\bal
\BBG_\mu[\tau](x)=\int_{\xO }G_{\mu}(x,y) \, \dd\tau(y), \quad \tau \in \GTM(\Omega;\ei), \\
 \mathbb{K}_\mu[\gn](x)=\int_{\partial\xO }K_{\mu}(x,y) \, \dd\xn(y), \quad \gn\in \mathfrak{M}(\partial\xO),
\eal
where $\GTM(A)$ and $\GTM_+(A)$ denote the space of Radon measures on $\partial\xO$ and its positive cone respectively for any Borel set $A\subset\BBR^N,$ as well as
\bal
\GTM(\Omega;\ei):=\bigg\{ \tau \in\GTM(\xO): \;\int_\xO\ei\,d|\xm|<\infty\bigg\}.
\eal

\subsection{The linear problem}

Let us recall the basic theory for the boundary value problem
\be\label{linear-bvp}
\left\{ \begin{array}{ll}
	-L_\mu u &= \tau \text{ in }\Omega\\
	\tr u&= \nu \\
\end{array}
\right..
\ee

\begin{definition}
Let $\tau\in\GTM(\Omega;\ei)$ and $\nu\in\GTM(\partial\Omega)$. We say that $u$ is a \textit{weak solution} of \eqref{linear-bvp} if $u\in L^1(\Omega;\ei)$ and 
\bal
-\int_\Omega uL_\mu \zeta\,dx = \int_\Omega \zeta\,d\tau - \vgr \int_\Omega \BBK_\mu[\nu]L_\mu\zeta\,dx \quad \forall \zeta \in \BBX_\mu(\Omega,\Sigma).
\eal
where
\bal
\BBX_\mu(\Omega,\Sigma):= \{ \zeta\in H^1_\loc(\Omega): \ei^{-1}\zeta\in H^1(\Omega;\ei^2) \text{ and } \ei^{-1}L_\mu\zeta \in L^\infty(\Omega) \}.
\eal
\end{definition}

This may seem somewhat technical and \textit{ad hoc}, but it is in fact inspired by the standard representation formula for such problems. In this case we have the following.

\begin{theorem}[{\cite[Theorem 2.12]{BGT}}]
Let $\tau\in\GTM(\Omega;\ei)$ and $\nu\in\GTM(\partial\Omega)$, and assume that $\lambda_\mu>0$. There exists a unique weak solution $u\in L^1(\Omega;\ei)$ of \eqref{linear-bvp}, namely
\bal
u=\BBG_\mu[\tau]+\BBK_\mu[\nu].
\eal
Furthermore there exists a positive constant $C=C(\Omega,\Sigma,\mu)$ such that
\bal
\| u \|_{L^1(\Omega;\ei)} \leq \frac{1}{\lambda_\mu} \| \tau \|_{\GTM(\Omega)}+ C \| \nu \|_{\GTM(\partial\Omega)}.
\eal
\end{theorem}

\subsection{Uniform integrability}

\begin{lemma} \label{UniformIntegrability}
Let $g:\R\times\R_+ \rightarrow \R_+$ be non-decreasing and locally Lipschitz in each of its variables, and assume that $g(0,0)=0$ and $$\int_1^\infty s^{-1-p}g(s,s^{p/q})\,ds < \infty$$ for some $p,q>0$. Let $u,v:\Omega\rightarrow\R$ be measurable functions, and for $s>0$ set $$E_w(s):=\{ x\in \Omega: |w(s)|>s \},\quad e_w(s):= \int_{E_w(s)}\ei\,dx,\quad w=u,v.$$ Finally, assume that there are positive constants $C_u,C_v$ such that for all $s>0$ $$e_u(s)\leq C_u s^{-p},\quad e_v(s)\leq C_v s^{-q}.$$ Then for any $s_0>0$ there holds
\bal
\| g(u,v) \|_{L^1(\Omega;\ei)} \leq \int_{E_u^c(s_0)\cap E_v^c(s_0^{p/q})} g(u,v)\ei\,dx +2p(C_u+C_v)\int_{s_0}^\infty s^{-1-p}g(s,s^{p/q})\,ds.
\eal
\end{lemma}

\begin{proof}
For $s>0$ set
\bal
e_{u,v}(s):= \int_{E_u(s)\cap E_v(s^{p/q})}\ei\,dx.
\eal
Then for each $s_0>0$ there holds
\bal
\int_\Omega g(u,v)\ei\,dx &= \int_{E_u^c(s_0)\cap E_v(s_0^{p/q})} g(u,v)\ei\,dx + \int_{E_u(s_0)\cap E_v^c(s_0^{p/q})} g(u,v)\ei\,dx \\
&\hspace{1cm} +\int_{E_u(s_0)\cap E_v(s_0^{p/q})} g(u,v)\ei\,dx + \int_{E_u^c(s_0)\cap E_v^c(s_0^{p/q})} g(u,v)\ei\,dx.
\eal
Now, observe that since $e_u(s)$ is the distribution function of $u$ relative to $\ei dx$, the properties of distribution functions imply
\bal
&\int_{E_u(s_0)\cap E_v^c(s_0^{p/q})} g(u,v)\ei\,dx \\
&\hspace{1cm} \leq -\int_{s_0}^\infty g(s,s_0^{p/q})\,de_u(s) \\
&\hspace{2cm} \leq pC_u \int_{s_0}^\infty g(s,s_0^{p/q})s^{-1-p}\,ds \\
&\hspace{3cm} \leq pC_u \int_{s_0}^\infty g(s,s^{p/q})s^{-1-p}\,ds
\eal
In the same way we show that
\bal
&\int_{E_u^c(s_0)\cap E_v(s_0^{p/q})} g(u,v)\ei\,dx \\
&\hspace{1cm} \leq -\int_{s_0}^\infty g(s_0,s^{p/q})\,de_v(s^{p/q}) \\
&\hspace{2cm} \leq pC_v \int_{s_0}^\infty g(s,s^{p/q})s^{-1-p}\,ds,
\eal
and 
\bal
&\int_{E_u(s_0)\cap E_v(s_0^{p/q})} g(u,v)\ei\,dx \\
&\hspace{1cm} \leq -\int_{s_0}^\infty g(s,s^{p/q})\,de_{u,v}(s) \\
&\hspace{2cm} \leq p\min\{C_u,C_v\} \int_{s_0}^\infty g(s,s^{p/q})s^{-1-p}\,ds.
\eal
Combining all the above, we obtain the desired result.
\end{proof}

\section{Weak Lebesgue estimates for the Green and Martin operators}\label{weak-lebesgue}

\subsection{Some preliminaries}

We denote by
$L^p_w(\Omega;\tau)$, $1 \leq p < \infty$, $\tau \in \GTM^+(\Omega)$, the
weak $L^p$ space (or Marcinkiewicz space) defined as follows:
a measurable function $f$ in $\Omega$
belongs to this space if there exists a constant $c$ such that
\[
\lambda_f(a;\tau):=\tau(\{x \in \Omega: |f(x)|>a\}) \leq ca^{-p},
\forevery a>0.
\]
The function $\lambda_f$ is called the distribution function of $f$ (relative to
$\tau$). For $p \geq 1$, denote
$$ L^p_w(\Omega;\tau)=\{ f \text{ Borel measurable}:
\sup_{a>0}a^p\lambda_f(a;\tau)<\infty\}, $$
\bel{semi}
\norm{f}^*_{L^p_w(\Omega;\tau)}=(\sup_{a>0}a^p\lambda_f(a;\tau))^{\frac{1}{p}}. \ee
This is not a norm, but for $p>1$, it is
equivalent to the norm
\[
\norm{f}_{L^p_w(\Omega;\tau)}=\sup\left\{
\frac{\int_{\omega}|f|d\tau}{\tau(\omega)^{1/p'}}:\omega \sbs \Omega, \omega \text{
measurable},\, 0<\tau(\omega)<\infty \right\}.
\]
More precisely,
\bel{equinorm} \norm{f}^*_{L^p_w(\Omega;\tau)} \leq \norm{f}_{L^p_w(\Omega;\tau)}
\leq \myfrac{p}{p-1}\norm{f}^*_{L^p_w(\Omega;\tau)}. \ee
When $d\tau=\ei dx$, for simplicity, we use the notation $L_w^p(\Omega;\ei)$.  Notice that,
$$
L_w^p(\Omega;\ei) \sbs L^{r}(\Omega;\ei), \quad  \forevery r \in [1,p).
$$
From \eqref{semi} and \eqref{equinorm} follows that for any
$u \in L_w^p(\Omega;\ei)$ there holds
\ba\label{intweakleb}
\int_{\{|u| \geq s\} }\ei dx \leq s^{-p}\norm{u}^p_{L_w^p(\Omega;\ei)}.
\ea

Let us recall \cite[Lemma 2.4]{BVi} which will be useful in the sequel.
\begin{proposition}
\label{bvivier}
Let $\omega$ be a nonnegative bounded Radon measure on $\partial \Omega$ and $\eta\in C(\xO)$ be a positive weight function. Let $\CH$ be a continuous nonnegative function
on $\xO\times \partial\Omega$. For $\xl > 0$ let
$$
A_\xl(y)=\{x\in\xO :\;\; \CH(x,y)>\xl\} \; , \quad  \quad
m_{\xl}(y)=\int_{A_\xl(y)}\eta(x)dx.
$$
Let $y\in\partial\Omega$ and
suppose that there exist $C>0$ and $\tau>1$ such that
$m_{\xl}(y)\leq C\xl^{-\tau}$ for every $\lambda>0$. Then the function
$$
\BBH[\omega](x):=\int_{\partial\Omega}\CH(x,y)d\omega(y)
$$
belongs to $L^\tau_w(\Omega;\eta )$ and
$$
\|\BBH[\omega]\|_{L^\tau_w(\Omega;\eta)}\leq
(1+\frac{C\tau}{\tau-1})\omega(\partial\Omega).
$$
\end{proposition}

\subsection{Weak $L^p$ estimates}

We first note that by scaling (see \cite[Lemma 3.2]{GN1}), we can easily show that there exists $C=C(\xm,N)$ such that
\ba
\label{estgradgreen}
|\nabla_x G_\xm(x,y)| &\leq C \frac{G_\xm(x,y)}{\min\{|x-y|,d(x)\}}\quad \forall x,y\in\xO,\;x\neq y\\ \nonumber
|\nabla_x K_\xm(x,\xi)| &\leq C \frac{K_\xm(x,\xi)}{d(x)}\quad \forall (x,\xi)\in\xO\times\partial\xO.
\ea
Set $$G_{\xm,\xg}(x,y)= \frac{G_\xm(x,y)}{\min\{|x-y|,d(x)\}^\xg}, \quad K_{\xm,\xg}(x,\xi)= \frac{K_{\xm}(x,\xi)}{d^\xg(x)},$$ where $\xg\in [0,1].$ In addition we set

\bal
\BBG_{\mu,\xg}[\tau](x)=\int_{\xO }G_{\mu,\xg}(x,y) \, \dd\tau(y), \quad \tau \in \GTM(\Omega;\ei), \\
 \mathbb{K}_{\mu,\xg}[\gn](x)=\int_{\partial\xO }K_{\mu,\xg}(x,y) \, \dd\xn(y), \quad \gn\in \mathfrak{M}(\partial\xO)
\eal

\begin{theorem} 
Assume $\lambda_\mu>0$, $0<\mu \leq H^2$, $0\leq\xg\leq1$ and let $$p:=\min \bigg\{ \frac{N+1}{N+\xg-1},\frac{N-\xa_-+1}{N-\xa_-+\xg-1} \bigg\}.$$
Then
\bal
	\norm{\BBG_{\mu,\xg}[\gt]}_{L_w^{p}(\Gw;\ei)} \lesssim \norm{\gt}_{\mathfrak{M}(\xO;\ei)}, \quad \forall \tau\in \mathfrak{M}(\xO;\ei).
\eal
	The implicit constant depends on $N,\Omega,\Sigma,\mu,p$.
\end{theorem}

\begin{proof}
We assume that $\xa_-<N/2$, and remark that the exceptional case $\xa_-=N/2$ is similar (as one can find an upper bound of the logarithmic term by a power term and proceed accordingly). By straightforward calculations and estimate \eqref{estgradgreen}, it follows that for all $x,y\in \xO$ such that $x\neq y$ and $|x-y| \neq d(x)$ there holds
\bal
G_{\xm,\xg} (x,y) &= \frac{G_\xm (x,y)}{(|x-y|\wedge d(x))^\xg} \\
	&\leq C |x-y|^{-N}\frac{d(x)d(y)\wedge |x-y|^2}{(d(x)\wedge |x-y|)^{\xg}} \bigg\{ \frac{(d_\xS(x)+|x-y|)(d_\xS(y)+|x-y|)}{d_\xS(x)d_\xS(y)} \bigg\}^{\xa_-}
\eal
where $C=C(\xm,N)$. It follows that
\ba\label{200}\BAL
G_{\xm,\xg} (x,y) \ei^{-1}(y) &\leq C |x-y|^{-N}\frac{d(x)d(y)\wedge |x-y|^2}{(d(x)\wedge |x-y|)^{\gamma}}(d_\Sigma(x)+|x-y|)^{2\alpha_-}d_\Sigma^{-\xa_-}(x)d^{-1}(y) \\
	&\lesssim F_1(x,y)+F_2(x,y)
\EAL,
\ea
where
\be
\label{F1green}
F_1(x,y):=|x-y|^{2-N} \bigg( \frac{1}{d(y)} \wedge \frac{d(x)}{|x-y|^2} \bigg) (d(x)\wedge |x-y|)^{-\gamma} d_\Sigma^{\xa_-}(x),
\ee
and
\be
\label{F2green}
F_2(x,y):=|x-y|^{2-N+2\xa_-} \bigg( \frac{1}{d(y)} \wedge \frac{d(x)}{|x-y|^2} \bigg) (d(x)\wedge |x-y|)^{-\gamma} d_\Sigma^{-\xa_-}(x).
\ee
Moreover, note that
\be
\label{DistQuotEst}
\frac{1}{d(y)}\wedge\frac{d(x)}{|x-y|^2} \leq \frac{4}{d(x)}\wedge \frac{d(x)}{|x-y|^2},
\ee
and our task now is to provide weak Lebesgue estimates for these terms close to $\partial \Omega$ and close to $y$. To this end, put
\bal
A_{\lambda,1}(y):= \{ x\in \Omega\setminus \{y\}: F_1(x,y)>\lambda\},\quad A_{\lambda,2}(y):= \{ x\in \Omega\setminus \{y\}: F_2(x,y)>\lambda\}.
\eal

\medskip

\paragraph{\textbf{Case 1a}} Note that if $x\in A_{\lambda,1}(y) \cap \{ d(x)<|x-y| \}$, then $$|x-y|< c \lambda^{-\frac{1}{N+\xg-1}}, \quad d_\Sigma^{-\xa_-}(x)< \lambda^{-1} |x-y|^{-N-\xg+1}$$ in view of \eqref{F1green} and \eqref{DistQuotEst}. It follows that
\bal
\int_{A_{\lambda,1}(y) \cap \{ d(x)<|x-y| \}} d(x)d_\Sigma^{-\xa_-}(x)\,dx \leq \lambda^{-1}\int_{\{|x-y|<c \lambda^{-\frac{1}{N+\xg-1}}\}} |x-y|^{2-N-\xg}\,dx \leq c\lambda^{-\frac{N+1}{N+\xg-1}}.
\eal

\medskip

\paragraph{\textbf{Case 1b}} Similarly, using the other branch of \eqref{DistQuotEst}, if $x\in A_{\lambda,1}(y) \cap \{ d(x)\geq |x-y| \}$ we have that $$|x-y|< c \lambda^{-\frac{1}{N+\xg-1}}, \quad d(x)< \frac{4}{\lambda} |x-y|^{2-N-\xg}d_\Sigma^{\xa_-}(x)$$ which implies
\bal
\int_{A_{\lambda,1}(y) \cap \{ d(x)\geq|x-y| \}} d(x)d_\Sigma^{-\xa_-}(x)\,dx \leq c\lambda^{-\frac{N+1}{N+\xg-1}}.
\eal

\medskip

\paragraph{\textbf{Case 2a}} If $x\in A_{\lambda,2}(y)\cap \{ d_\Sigma(x) <|x-y| \}$, by \eqref{F2green} and \eqref{DistQuotEst} we see that $$d_\Sigma(x)< \lambda^{-\frac{1}{N-\xa_-+\xg-1}},\quad |x-y|<\lambda^{-\frac{1}{N-2\xa_-}} d_\Sigma^{\frac{1-\xa_--\xg}{N-2\xa_-}}(x),$$ and thus, by Lemma \ref{lemapp:1}, we obtain
\bal
\int_{A_{\lambda,2}(y)\cap \{ d_\Sigma(x) <|x-y| \} \cap \Sigma_{\beta_1}}d(x)d_\Sigma^{-\xa_-}(x)\,dx \leq c \lambda^{-\frac{N-\xa_-+1}{N-\xa_-+\xg-1}}.
\eal
Finally, $A_{\lambda,2}(y)\cap \{ d_\Sigma(x) <|x-y| \} \setminus \Sigma_{\beta_1} = \varnothing$ for $\lambda>\beta_1^{-N+\xa_--\xg+1}$.

\medskip

\paragraph{\textbf{Case 2b}} If $x\in A_{\lambda,2}(y) \cap \{ d_\Sigma(x)\geq |x-y| \}$, we distinguish two further cases. First let $d(x)<|x-y|$, in which case
\bal
|x-y|<\lambda^{-\frac{1}{N-\xa_-+\xg-1}}.
\eal
It follows that
\bal
&\hspace{-1cm}\int_{A_{\lambda,2}(y) \cap \{ d_\Sigma(x)\geq |x-y| \}\cap\{d(x)<|x-y|\}} d(x)d_\Sigma^{-\xa_-}(x)\,dx\\
&\leq \int_{\{|x-y|<\lambda^{-\frac{1}{N-\xa_-+\xg-1}} \}} |x-y|^{1-\xa_-}\,dx\\
&\hspace{1cm}\leq c\lambda^{-\frac{N-\xa_-+1}{N-\xa_-+\xg-1}}.
\eal
If, on the other hand, $|x-y|\leq d(x)$, we see that
\bal
|x-y|<(4/\lambda)^{\frac{1}{N-\xa_-+\xg-1}},\quad d(x)<\frac{4}{\lambda}|x-y|^{2-N+\xa_--\xg},
\eal
so in this case
\bal
&\hspace{-1cm}\int_{A_{\lambda,2}(y) \cap \{ d_\Sigma(x)\geq |x-y| \}\cap\{d(x)\geq|x-y|\}} d(x)d_\Sigma^{-\xa_-}(x)\,dx\\
&4\lambda^{-1}\int_{\{|x-y|<(4/\lambda)^{\frac{1}{N-\xa_-+\xg-1}} \}}|x-y|^{2-N-\xg}\,dx\\
&\hspace{1cm}\leq c\lambda^{-\frac{N-\xa_-+1}{N-\xa_-+\xg-1}}
\eal
as well. 

Setting $\mathbb{F}_i[\ei \tau](x)=\int_{\xO}F_i(x,y)\ei d \tau(y)$, we apply Proposition \ref{bvivier} with {\small$\mathcal{H}(x,y)=F_i(x,y)$}, $\eta=\ei$ and $\omega=\ei\tau$ to obtain
\bal
\norm{\mathbb{F}_i[\nu]}_{L_w^{p}(\Omega;\ei)} \leq
C\norm{\tau}_{\mathfrak{M}(\xO;\ei)},\quad\text{for}\;i=1,2.
\eal
Combining the above estimate and \eqref{200}, we obtain the desired result.
\end{proof}

\begin{theorem}
\label{lpweakmartin1}
Assume that $\mu\leq H^2$ and $\lambda_{\mu}>0$, and let $p:=\min \big\{ \frac{N+1}{N+\xg-1},\frac{N-\xa_-+1}{N-\xa_-+\xg-1} \big\}$. Then there exists a positive constant $C=C(\Omega,\xS,\mu)$ such that
\[
\norm{\mathbb{K}_{\mu,\xg}[\nu]}_{L_w^{p}(\Omega;\ei)} \leq
C\norm{\nu}_{\mathfrak{M}(\partial\xO)}
\]
for any measure $\nu\in \mathfrak{M}(\partial\xO)$.
\end{theorem}

\begin{proof}
 Without loss of generality, we may assume that $\nu\in \mathfrak{M}_+(\partial \Omega)$. Let $\lambda>0$ and $y\in \partial\xO$. We only consider the case  $0<\xm<\big( \frac{N}{2}\big)^2$. The other cases can be treated similarly and we omit them. Since in that case $\xa_->0$, from \eqref{Martinest1} we have that
\[
K_{\xm,\xg}(x,y) \approx \frac{d^{1-\xg}(x)d_\Sigma^{\xa_-}(x)}{|x-y|^N}+\frac{d^{1-\xg}(x)d_\Sigma^{-\xa_-}(x)}{|x-y|^{N-2\xa_-}}\quad \mbox{ in }\xO\times\partial\xO.
\]
Set
\[
F_1(x,y):=\frac{d^{1-\xg}(x)d_\Sigma^{\xa_-}(x)}{|x-y|^N} ,\quad\quad F_2(x,y):=\frac{d^{1-\xg}(x)d_\Sigma^{-\xa_-}(x)}{|x-y|^{N-2\xa_-}}.
\]

\medskip

\paragraph{\textbf{Case 1a}} Let
\begin{align*}
A_\xl(y):=\big\{x\in \xO:\;F_2(x,y)>\xl \big\}, \quad\quad
m_{\xl}(y):=\int_{A_\xl(y)}d(x)d_\Sigma^{-\xa_-}(x) dx.
\end{align*}

Note that if $x\in A_\xl(y)\cap\{d_\Sigma(x)> |x-y|\}$, by taking into account that $d(x)\leq |x-y|$ we see that $$|x-y|<\xl^{-\frac{1}{N-\xa_- +\xg-1}}$$
in that case, thus
\[
\int_{A_\xl(y)\cap\{d_\Sigma(x)> |x-y|\}}d(x)d_\Sigma^{-\xa_-}(x) dx\leq \int_{\{|x-y|<\xl^{-\frac{1}{N-\xa_-+\xg -1}}\}}|x-y|^{1-\xa_-}dx=C
\xl^{-\frac{N-\xa_-+1}{N-\xa_-+\xg -1}}.
\]

If, on the other hand, $x\in A_\xl(y)\cap\{d_\Sigma(x)\leq |x-y|\}$, taking into account that $d_\Sigma(x)\geq d(x)$, we see that
\bal
|x-y| < \lambda^{-\frac{1}{N-2\xa_-}} d_\Sigma^{\frac{1-\xg-\xa_-}{N-2\xa_-}}(x),
\eal
and since $d_\Sigma(x)\leq |x-y|$, in that case we have
\bal
d_\Sigma(x) \leq \lambda^{-\frac{1}{N-\xa_-+\xg-1}}.
\eal
Taking into account Lemma \ref{lemapp:1}, it follows that
\bal
 \int_{A_\xl(y)\cap\{d_\Sigma(x)\leq |x-y|\}\cap \Sigma_{\beta_1}}d(x)d_\Sigma^{-\xa_-}(x) dx \leq C\lambda^{-\frac{N-\xa_-+1}{N-\xa_-+\xg-1}}.
\eal

Now let $x\in \xO\setminus \Sigma_{\xb_1}$. Since $d_\Sigma(x)>\xb_1$ in that case, it follows that $A_\xl(y)\cap\{d_\Sigma(x)\leq |x-y|\}\setminus B = \varnothing$ provided that $\lambda$ is large enough. 

Hence, combining these estimates  we obtain the final estimate
\bal
m_{\xl}(y) \leq C \xl^{-\frac{N-\xa_-+1}{N-\xa_-+\xg-1}}.
\eal
Setting $\mathbb{F}_2[\xn](x)=\int_{\partial\xO}F_2(x,y)d\xn(y)$, we apply Proposition \ref{bvivier} with $\mathcal{H}(x,y)=F_2(x,y)$, $\eta=\ei$ and  $\omega=\xn$ to obtain
\be\label{F2weak estimate}
\norm{\mathbb{F}_2[\nu]}_{L_w^{\frac{N-\xa_-+1}{N-\xa_-+\gamma-1}}(\Omega;\ei)} \leq
C\norm{\nu}_{\mathfrak{M}(\partial\xO)}.
\ee

\medskip

\paragraph{\textbf{Case 1b}} Now let
\begin{align*}
\tilde A_\xl(y):=\big\{x\in \xO:\;F_1(x,y)>\xl \big\}, \quad\quad
\tilde m_{\xl}(y):=\int_{A_\xl(y)}d(x)d_\Sigma^{-\xa_-}(x) dx.
\end{align*}
If $x\in \tilde A_\lambda(y)$, it is immediate that $d_\Sigma^{\xa_-}(x)>\lambda |x-y|^{N-1+\xg}$ in view of $d(x)\leq |x-y|$, by which we additionally derive
\bal
|x-y|<C\lambda^{-\frac{1}{N-1+\xg}} \quad \textrm{and}\quad d_\Sigma^{-\xa_-}(x)<\lambda^{-1}|x-y|^{-N+1-\xg},
\eal
the first one owing to $d_\Sigma(x)\leq \mathrm{diam}(\Omega)$ and $\xa_->0$. It follows that
\begin{align*}
\tilde m_\lambda(y) &\leq \int_{\{|x-y|<C\lambda^{-\frac{1}{N-1+\xg}}\}}\xl^{-1}|x-y|^{2-N-\xg}dx\\
	&\leq C \lambda^{-\frac{N+1}{N-1+\xg}},
\end{align*}
and thus, setting $\mathbb{F}_1[\xn](x)=\int_{\partial\xO}F_1(x,y)d\xn(y)$ we obtain, in view of Proposition \ref{bvivier},
\be
 \label{F1weak estimate}
\norm{\mathbb{F}_1[\nu]}_{L_w^{\frac{N+1}{N-1+\xg}}(\Omega;\ei)} \leq C\norm{\nu}_{\mathfrak{M}(\partial\xO)}.
\ee
Estimates \eqref{F2weak estimate} and \eqref{F1weak estimate} imply the conclusion.

\medskip

This completes the proof.
\end{proof}

\begin{theorem}\label{martin-est-p1-p2} ~~
Let $\lambda_\mu>0$ and $p_1:=\frac{N+1}{N+\xg-1}$ and $p_2:=\frac{N-\xa_-+1}{N-\xa_-+\xg-1}$. Then the following statements hold:
\begin{enumerate}
\item Assume $\mu \leq \left( \frac{N}{2}\right)^2$ and $\gn\in \mathfrak{M}(\partial\xO)$ with compact support in $\partial\xO\setminus\xS.$ Then there exists a positive constant $C=C(\xO,\xS,\xm,\supp \xn,\xg)$ such that
\bal
	\norm{\mathbb{K}_{\mu,\xg}[\nu]}_{L_w^{p_1}(\Gw;\ei)} \leq C \|\nu\|_{\mathfrak{M}(\partial\Omega )}.
\eal
	
\item Assume that $\mu < \left( \frac{N}{2} \right)^2$ and $\gn\in \mathfrak{M}(\partial\xO)$ with compact support in $\Sigma$. Then there exists a positive constant $C=C(N,\Omega,\Sigma,\mu,\xg)$ such that
\bal
\norm{\mathbb{K}_{\mu,\xg}[\nu]}_{L_w^{p_2}(\Gw;\ei)} \lesssim \norm{\nu}_{\mathfrak{M}(\partial \Omega )}.
\eal

\item Assume that $\mu = \left( \frac{N}{2} \right)^2$. For any $0<\xe<p_2-1,$ there exists a positive constant $C=C(N,\Omega,\Sigma,N,\xe)$ such that
\bal
\norm{\mathbb{K}_{\mu,\xg}[\xd_0]}_{L^{p_2-\xe}(\Gw;\ei)} \lesssim 1.
\eal
\end{enumerate}
\end{theorem}

\begin{proof}
This is an extension of \cite[Theorem B.2]{BGT}, where the case $\xg=0$ is treated. The proof is similar to that of Theorem \ref{lpweakmartin1}; see also \cite[Theorem 3.11]{GkiNg_source} for a detailed proof of the case $\Sigma\subset\Omega$.
\end{proof}

\section{Nonlinear equations with subcritical source} \label{subctitical-problem}

For convenience, we use special notation for the critical exponents
\bal
p_*:= \min \bigg\{ \frac{N+1}{N-1},\frac{N-\xa_-+1}{N-\xa_--1} \bigg\},\quad q_*:= \min \bigg\{ \frac{N+1}{N},\frac{N-\xa_-+1}{N-\xa_-} \bigg\}
\eal
corresponding to $\gamma=0$ and $\gamma=1$ respectively, to be used in the sequel.

In whatever follows, we assume that $\lambda_\mu>0$ and that $\mu\leq H^2$. Moreover, $g:\BBR\times\BBR_+\to\BBR_+$ is assumed to be continuous and non-decreasing in its arguments with $g(0,0)=0$. First we provide an auxiliary existence result for continuous and bounded $g$.

\begin{lemma}\label{SubcriticalLemma}
Let $\nu\in \GTM_+(\partial\Omega)$ with $\|\nu\|_{\GTM(\partial\Omega)}=1$ and $g\in C(\R\times\R_+)\cap L^\infty(\R\times\R_+)$ be such that
\bal
\Lambda_g:= \int_1^\infty g(s,s^{p_*/q_*})s^{-1-p_*}\,ds < \infty \qquad \text{and} \qquad g(as,bt)\leq c(a^p+b^q)g(s,t)
\eal
for some $p,q>1$, $c>0$ and all $a,b,s,t\in \R_+$. Then there exists a positive $\vgr_0=\vgr(M,\mu,\Omega,\xL_g,c,p,q)$ such that the problem
\be\label{aux1}
\left\{ \begin{array}{ll}
	-L_\mu w &= g(w+\vgr \BBK_\mu[\nu],|\nabla(w+\vgr \BBK_\mu[\nu])|) \text{ in }\Omega\\
	\tr(w)&=0\\
\end{array}
\right.
\ee
possesses a positive weak solution for all $\vgr\in (0,\vgr_0)$, which satisfies
\be\label{aux1-estimate}
\|w\|_{L^{p_*}_w(\Omega;\ei)} + \|\nabla w\|_{L^{q_*}_w(\Omega;\ei)} \leq t_0
\ee
for some positive $t_0=t_0(N,\mu,\Omega,\Lambda_g,c,p,q)$.
\end{lemma}

\begin{proof}
To demonstrate the existence of a positive weak solution of \eqref{aux1}, we define the operator
\bal
\BBA[w]:= \BBG_\mu[g(w+\vgr \BBK_\mu[\nu],|\nabla(w+\vgr \BBK_\mu[\nu])|)],\quad w\in W^{1,1}(\Omega;\ei),
\eal
with the intention of applying the Schauder fixed point theorem. Here we denote 

\bal
W^{1,1}(\Omega;\ei):= \{ v\in W^{1,1}_{\loc}(\Omega): v,|\nabla v|\in L^{1}(\Omega;\ei)\}.
\eal

Fix $\kappa\in (1,\min\{p,p_*,q,q_*\})$ and set
\bal
A_1(w)&:= \|w\|_{L^{p_*}_w(\xO,\ei)},\;\; A_2(w):=\|\nabla w \|_{L^{q_*}_w(\xO,\ei)},\;\; A_3(w):= \|w\|_{L^\xk(\xO,\ei)}\\
 A_4(w)&:= \|\nabla w\|_{L^\xk(\xO,\ei)}\quad\text{and}\quad A(w) := A_1(w)+A_2(w)+A_3(w)+A_4(w).
\eal 

\medskip

\paragraph{\textbf{Step 1}} First we need an estimate of $\| g(w+\vgr \BBK_\mu[\nu],|\nabla(w+\vgr \BBK_\mu[\nu])|) \|_{L^1(\xO,\ei)}$. Set
\bal
A_{\lambda}&:=\{x\in\Omega: |w+\vgr \BBK_\mu[\nu]|(x)>\lambda \},\quad B_\lambda:= \{x\in\Omega: |\nabla(w+\vgr \BBK_\mu[\nu])|(x)>\lambda^{\frac{p_*}{q_*}} \},\\ 
C_\lambda &:= A_\lambda \cap B_\lambda,
\eal
and moreover define the distribution functions
\bal
a(\lambda):= \int_{A_\lambda} \ei\,dx,\quad b(\lambda):=\int_{B_{\lambda}} \ei\,dx,\quad c(\lambda):= \int_{C_\lambda} \ei\,dx.
\eal
Then, in view of \eqref{intweakleb} we see that
\bal
&a(\lambda) \leq \lambda^{-p_*} \| w+\vgr \BBK_\mu[\nu] \|_{L^{p_*}_w(\Omega,\ei)}^{p_*},\\
&b(\lambda) \leq \lambda^{-p_*} \| \nabla(w+\vgr \BBK_\mu[\nu]) \|_{L^{q_*}_w(\Omega,\ei)}^{q_*},\\
&c(\lambda) \leq \lambda^{-p_*} \min \{ \| w+\vgr \BBK_\mu[\nu] \|_{L^{p_*}_w(\Omega,\ei)}^{p_*}, \| \nabla(w+\vgr \BBK_\mu[\nu] )\|_{L^{q_*}_w(\Omega,\ei)}^{q_*} \}.
\eal
With this notation, we split the integral
\bal
&\| g(w+\vgr \BBK_\mu[\nu],|\nabla(w+\vgr \BBK_\mu[\nu])|) \|_{L^1(\xO,\ei)} \leq \int_{C_1} g(\cdots,|\nabla \cdots|)\ei\,dx + \int_{A_1\cap B_1^c} g(\cdots,|\nabla \cdots|)\ei\,dx \\
&\qquad+\int_{A_1^c\cap B_1^c} g(\cdots,|\nabla \cdots|)\ei\,dx + \int_{A_1^c\cap B_1} g(\cdots,|\nabla \cdots|)\ei\,dx\\
&\qquad=: I_1+I_2+I_3+I_4.
\eal
Integral $I_1$ is estimated as follows. By the properties of distribution functions we see that
\bal
I_1 &\leq -\int_1^\infty g(\lambda,\lambda^{\frac{p_*}{q_*}})\,dc(\lambda) = -g(1,1)c(1) +\int_1^\infty c(\lambda)\, dg(\lambda,\lambda^{\frac{p_*}{q_*}}) \\
&\leq p_* \min \{ \| w+\vgr \BBK_\mu[\nu] \|_{L^{p_*}_w(\Omega,\ei)}^{p_*}, \| \nabla(w+\vgr \BBK_\mu[\nu]) \|_{L^{q_*}_w(\Omega,\ei)}^{q_*} \} \int_1^\infty g(\lambda,\lambda^{p_*/q_*})\lambda^{-1-p_*}\,d\lambda.
\eal
Integrals $I_2$ and $I_4$ are treated similarly. For $I_2$ we have 
\bal
I_2 \leq -\int_1^\infty g(\lambda,1)\,da(\lambda) \leq p_* \| w+\vgr \BBK_\mu[\nu] \|_{L^{p_*}_w(\Omega,\ei)}^{p_*} \int_1^\infty g(\lambda,\lambda^{p_*/q_*})\lambda^{-1-p_*}\,d\lambda,
\eal
and likewise 
\bal
I_4 \leq p_* \|\nabla(w+\vgr \BBK_\mu[\nu]) \|_{L^{q_*}_w(\Omega,\ei)}^{q_*} \int_1^\infty g(\lambda,\lambda^{p_*/q_*})\lambda^{-1-p_*}\,d\lambda.
\eal
As for $I_3$, observe that the power condition on $g$ and the fact that $1<\kappa<\min\{ p,p_*,q,q_* \}$ implies 
\bal
I_3\leq cg(1,1)(\| w+\vgr \BBK_\mu[\nu] \|_{L^{k}(\Omega,\ei)}^{k}+\| \nabla(w+\vgr \BBK_\mu[\nu]) \|_{L^{k}(\Omega,\ei)}^{k}).
\eal
Invoking the weak Lebesgue estimates for $\BBK_{\mu,\gamma}$ with $\gamma=0,1$ we conclude that
\bal
\| g(w+\vgr \BBK_\mu[\nu],|\nabla(w+\vgr \BBK_\mu[\nu])|) \|_{L^1(\xO,\ei)} \leq C(A_1^{p_*}(w)+A_2^{q_*}(w)+A_3^\xk(w)+A_4^\xk(w)+\vgr^\xk)
\eal
for some positive constant $C=C(N,\mu,\Omega,\Lambda_g,c)$.

\medskip

\paragraph{\textbf{Step 2}} We estimate
\bal
A_1(\BBA[w]) &= \| \BBG_\mu[g(w+\vgr \BBK_\mu[\nu],|\nabla(w+\vgr \BBK_\mu[\nu])|)] \|_{L^{p_*}_w(\xO,\ei)}\\
	&\lesssim \| g(w+\vgr \BBK_\mu[\nu],|\nabla(w+\vgr \BBK_\mu[\nu])|) \|_{L^1(\xO,\ei)}
\eal
and
\bal
A_2(\BBA[w]) &= \| \nabla \BBG_\mu[g(w+\vgr \BBK_\mu[\nu],|\nabla(w+\vgr \BBK_\mu[\nu])|)] \|_{L^{q_*}_w(\xO,\ei)}\\
	&\lesssim \| g(w+\vgr \BBK_\mu[\nu],|\nabla(w+\vgr \BBK_\mu[\nu])|) \|_{L^1(\xO,\ei)}
\eal
and similarly
\bal
A_3(\BBA[w])+A_4(\BBA[w]) \lesssim \| g(w+\vgr \BBK_\mu[\nu],|\nabla(w+\vgr \BBK_\mu[\nu])|) \|_{L^1(\xO,\ei)}.
\eal
Now if $A(w)\leq t$, then $A_i(w)\leq t$ for $i=1,\ldots,4$, so it follows that
\bal
A(\BBA[w]) \leq C(t^{p_*}+t^{q_*}+2t^\xk+\vgr^\xk).
\eal
Note that the right-hand side is strictly convex and has value $\vgr^k$ at zero. It follows that there exist  positive $\vgr_0=\vgr_0(N,\mu,\Omega,\Lambda_g,c)$ and $t_0=t_0(N,\mu,\Omega,\Lambda_g,c)$ such that for any $\vgr\in(0,\vgr_0)$, $t\in(0,t_0)$ the inequality
\bal
C(t^{p_*}+t^{q_*}+2t^\xk+\vgr^\xk) \leq t_0
\eal
holds. It follows that
\bal
A(w)\leq t_0 \Rightarrow A(\BBA[w]) \leq t_0.
\eal

\medskip

\paragraph{\textbf{Step 3}} We prove that $\BBA$ is continuous. Indeed, if $w_n\rightarrow w$ in $W^{1,1}(\Omega;\ei)$ then since $g\in C(\BBR\times\BBR_+) \cap L^\infty(\BBR\times\BBR_+)$ it follows that
\bal
g(w_n+\vgr\BBK[\nu],|\nabla(w_n+\vgr\BBK[\nu])|) \longrightarrow g(w+\vgr\BBK[\nu],|\nabla(w_n+\vgr\BBK[\nu])|) \text{ in } L^1(\Omega,\ei),
\eal
and thus $\BBA[w_n] \rightarrow \BBA[w]$ in $W^{1,1}(\Omega;\ei)$ by the continuity of $\BBG_\mu$ and $\nabla\BBG_\mu$.

\medskip

\paragraph{\textbf{Step 4}} We prove that $\BBA$ is compact. Setting $M:=\sup |g| <\infty$, by the previous discussion we easily see that for all $w\in W^{1,1}(\Omega;\ei)$ there holds
\bal
\| \BBA[w] \|_{W^{1,1}(\Omega;\ei)} \leq C(M,\Omega,\mu),
\eal
so it follows that for any sequence $\{w_n\} \subset W^{1,1}(\Omega;\ei)$, the image $\{\BBA[w_n]\}$ is unifomly bounded in $W^{1,1}(\Omega;\ei)$. On the other hand, local elliptic regularity implies that  $\{\BBA[w_n]\}$ is uniformly bounded in $W^{2,1}(D)$ for any $D\Subset\xO$ and hence possesses, up to a subsequence, a strong limit in $W^{1,1}_\loc(\Omega)$, which is also the desired global limit in $W^{1,1}(\Omega;\ei)$ in view of the dominated convergence theorem.

\medskip

\paragraph{\textbf{Step 5}} Combining all previous steps we apply the Schauder fixed point theorem. Set
\bal
\CO:= \{w\in W^{1,1}(\Omega;\ei): A(w)\leq t_0 \},
\eal
which is easily seen to be closed and convex. Then $\BBA[\CO]\subset\CO$ and we have already proven that $\BBA$ is continuous and compact. This yields a solution of \eqref{aux1} satisfying the desired estimate. 
\end{proof}

\begin{proof}[\textbf{Proof of Theorem \ref{subcritical}}]
Let $g_n=\min(n,g)$ for any $n\in\BBN$. From the definition of $\Lambda_g$ it is obvious that $\Lambda_{g_n}\leq \Lambda_g$. Then, considering the auxiliary problem \eqref{aux1} with $g$ replaced by $g_n$, one can find uniform positive constants $\vgr_0$ and $t_0$ which are independent of $n$ so that the conclusions of Lemma \ref{SubcriticalLemma} are valid; we denote the respective solutions by $w_n$. Setting $u_n=w_n+\vgr\BBK[\nu]$, we see that $u_n$ solves the problem
\bal
\left\{ \begin{array}{ll}
	-L_\mu u &= g_n(u,|\nabla u |) \text{ in }\Omega\\
	\tr(u)&= \vgr\nu\\
\end{array}
\right.
\eal
provided that $\vgr\in(0,\vgr_0)$, so in particular
\be\label{subcrit-weak-n}
-\int_\Omega u_n L_\mu \zeta\,dx = \int_\Omega g_n(u_n,|\nabla u_n|)\zeta\,dx -\vgr \int_\Omega \BBK[\nu]L_\mu \zeta\,dx \quad \forall \zeta \in \BBX_\mu(\Omega,\Sigma)
\ee

Since $\{w_n\} \subset \CO$, it follows that $\{ g_n(w_n+\vgr\BBK[\nu],|\nabla(w_n+\vgr\BBK[\nu])|)\}$ is bounded in $L^1(\Omega;\ei)$. Moreover, $\{ \mu w_n d_\Sigma^{-2} \}$ is bounded in $L^1_\loc(\Omega)$, and thus $\{ \Delta w_n \}$ is bounded in $L^1_\loc(\Omega)$, and by local elliptic regularity it follows that $\{w_n \}$ is bounded in $W^{2,1}_\loc(\Omega)$. This implies that, up to a subsequence, $w_n\rightarrow w$ strongly in $W^{1,1}_\loc(\Omega)$ for some $w$, and therefore
\bal
w_n \longrightarrow w \text{ \ a.e.},\quad \nabla w_n \longrightarrow \nabla w \text{ \ a.e. \ in }\Omega.
\eal

By Lemma \ref{UniformIntegrability} and estimate \eqref{aux1-estimate} we deduce the existence of a positive constant $C = C(\Omega,\mu,\Lambda_g,c,p,q,t_0)$ such that for $\lambda\geq 1$
\bal
\int_B g_n(u_n,|\nabla u_n|)\ei\,dx \leq C \int_\lambda^\infty g(s,s^{p_*/q_*})s^{-1-p_*}\,ds + g(\lambda,\lambda^{p_*/q_*}) \int_B \ei\,dx \quad \forall B\in \CB(\Omega).
\eal
Note that the first term on the right-hand side becomes arbitrarily small as $\lambda\rightarrow \infty$, so for any $\vge>0$ there exists a $\lambda>1$ such that
\bal
C \int_\lambda^\infty g(s,s^{p_*/q_*})s^{-1-p_*}\,ds <\frac{\vge}{2}.
\eal
Fixing such $\lambda$ and setting $\delta = \vge / 2\max\{g(\lambda,\lambda^{p_*/q_*}), 1\}$ it follows that
\bal
\int_B \ei\,dx< \delta \Rightarrow \int_B g_n(u_n,|\nabla u_n|)\,dx <\vge \quad \forall B\in \CB(\Omega),
\eal
therefore $\{ g_n(u_n,|\nabla u_n|) \}$ is equi-integrable in $L^1(\Omega;\ei)$. In view of Vitali's convergence theorem, we conclude that
\bal
g_n(u_n,|\nabla u_n|) \longrightarrow g(u,|\nabla u|) \text{ \ in } L^1(\Omega,\ei).
\eal
Then we can let $n\rightarrow \infty$ in \eqref{subcrit-weak-n} to conclude that
\bal
-\int_\Omega u L_\mu \zeta\,dx = \int_\Omega g(u,|\nabla u|)\zeta\,dx -\vgr \int_\Omega \BBK[\nu]L_\mu \zeta\,dx \quad \forall \zeta \in \BBX_\mu(\Omega,\Sigma).
\eal
This means that $u$ is a weak non-negative solution of \eqref{subcrit-bvp}, satisfying
\bal
u= \BBG_\mu[g(u,|\nabla u|)]+\vgr\BBK[\nu].
\eal
In particular we have $u\geq \vgr\BBK[\nu]>0$ in $\Omega$. This completes the proof.
\end{proof}

\begin{remark}
Let us mention briefly the two main examples of subcritical sources that we are interested in. Both examples are of power-type, so condition $g(as,bt)\leq c(a^p+b^q)g(s,t)$ is fullfilled automatically and it suffices to consider only the condition $\Lambda_g<\infty$.
\begin{itemize}
\item Let $g(u,|\nabla u|) = a|u|^p+b|\nabla u|^q$ for some $a,b\geq 0$. Then $\Lambda_g<\infty$ if and only if $p<p_*$ and $q<q_*$.
\item Let $g(u,|\nabla u|) = c|u|^p |\nabla u|^q$ for some $c>0$. Then $\Lambda_g<\infty$ if and only if $p/p_* + q/q_* <1$.
\end{itemize}
\end{remark}

\section{Nonlinear equations with supercritical source} \label{supercritical-source}

\subsection{Abstract setting}

Let $M$ be a metric space and let $J:M\times M \rightarrow (0,\infty]$ be a positive Borel kernel such that $J(x,y)=J(y,x)$ and $J^{-1}$ satisfies a quasi-metric inequality, i.e.
\bal
J^{-1}(x,y)\leq C (J^{-1}(x,z)+J^{-1}(z,y))\quad \forall x,y\in M
\eal
for some $C>1$. In that case we may define the quasi-metric $\mathsf{d}(x,y) := J^{-1}(x,y)$ and denote by
\bal
\mathfrak{B}(x,r):= \{ y\in M: \mathfrak{d}(y,x)<r\}
\eal
the open $\mathfrak{d}$-ball of radius $r$ and center $x$; note that this can be empty, as in general we may have $\mathfrak{d}(x,x)\neq 0$.

Now let $\omega\in \GTM_+(M)$ and $\phi\geq 0$ be a non-negative Borel function. We define the potentials
\bal
\BBJ[\omega](x) \equiv \BBJ[d\omega](x):= \int_M J(x,y)\,d\omega(y),\quad \BBJ[\phi d\omega](x):=\int_M J(x,y)\phi(y)\, d\omega(y),
\eal
as well as the $(p,J,\omega)$-capacity
\bal
\mathrm{Cap}^{p}_{J,\omega}(B):= \inf \bigg\{ \int_M \phi^p\,d\omega : \phi\geq 0,\ \BBJ[\phi d\omega]\geq \chi_B \bigg\}\quad \forall B\in \CB(M)
\eal
where $p>1$.

In this setting we have the following proposition:

\begin{proposition}[\cite{KV}]\label{abstract-proposition}
Let $p>1$ and $\lambda,\omega\in \BBM_+(M)$ be such that
\be\label{abstract-condition-1}
\int_0^{2r} \frac{\omega(\mathfrak{B}(x,s))}{s^2}\,ds \leq C \int_0^r \frac{\omega(\mathfrak{B}(x,s))}{s^2}\,ds
\ee
\be\label{abstract-condition-2}
\sup_{y\in\mathfrak{B}(x,r)} \int_0^r \frac{\omega(\mathfrak{B}(y,s))}{s^2}\,ds \leq C \int_0^r \frac{\omega(\mathfrak{B}(x,s))}{s^2}\,ds
\ee
for any $r>0$ and $x\in M$, where $C>0$ is a constant. Then the following statements are equivalent:
\begin{enumerate}
\item The equation
\bal
v=\BBJ[|v|^p d\omega]+ l \BBJ[\lambda]
\eal
possesses a positive solution provided that $l>0$ is small enough.
\item For any $B\in \CB(M)$ there holds
\bal
\int_B \BBJ[\chi_B \lambda]^p\, d\omega \leq C \lambda(B).
\eal
\item The following inequality holds:
\bal
\BBJ[\BBJ[\lambda]^pd\omega] \leq C \BBJ[\lambda]< \infty.
\eal
\item For any $B\in \CB(M)$ there holds
\bal
\lambda(B)\leq C\, \mathrm{Cap}_{J,\omega}^{p'}(B).
\eal
\end{enumerate}
\end{proposition}

\subsection{Necessary and sufficient conditions for existence}

For $\xa\leq N$ and $0<\xs<N$ we set
\bal
\CN_{\xa,\xs}(x,y) &:= \frac{\max\{ |x-y|,d_\Sigma(x),d_\Sigma(y) \}^\xa}{|x-y|^{N-\xs} \max\{ |x-y|,d(x),d(y) \}^\xs} \quad \forall x,y\in \overline{\Omega},\, x\neq y,\\
\BBN_{\xa,\xs}[\omega](x) &:= \int_{\overline{\Omega}} \CN_{\xa,\xs}(x,y)\, d\omega(y) \quad \forall \omega\in \GTM_+(\overline{\Omega}),
\eal
with the intention of applying the abstract setting of the previous paragraph.

\begin{lemma}
Let $\alpha \leq N$ and $0<\sigma<N$. There exists a positive constant $C=C(\Omega,\Sigma,\xa,\xs)$ such that
\be\label{quasi-norm}
\CN_{\xa,\xs}^{-1}(x,y) \leq C (\CN_{\xa,\xs}^{-1}(x,z)+\CN_{\xa,\xs}^{-1}(z,y)) \quad \forall x,y,z\in \overline{\Omega}.
\ee
\end{lemma}

\begin{proof}
We proceed as in \cite[Lemma 6.3]{GN2} with minor modifications. We will consider two cases, for positive and non-positive $\xa$.

\medskip

\paragraph{\textbf{Case 1}} Let $0<\xa\leq N$, and suppose for a moment that $|x-y|<2|x-z|$. By the triangle inequality we have $d_\Sigma(z)\leq |x-z|+d_\Sigma(x) \leq 2 \max \{ |x-z|,d_\Sigma(x) \}$, hence
\bal
\max \{ |x-z|,d_\Sigma(x),d_\Sigma(z) \} \leq 2 \max \{ |x-z|,d_\Sigma(x) \}
\eal
If $|x-z|\geq d_\Sigma(x)$ then $|x-z|\geq d(x)$, so in that case $|x-z|\geq (d(x)+|x-y|)/4$, and then
\bal
&\hspace{-1cm}\CN_{\xa,\xs}^{-1}(x,z)  = \frac{|x-z|^{N-\xs}\max \{ |x-z|,d(x),d(z) \}^\xs}{\max\{ |x-z|,d_\Sigma(x),d_\Sigma(z) \}^\xa} \gtrsim |x-z|^{N-\xa} \gtrsim (|x-y|+d(x))^{N-\xa}\\
&= \frac{(|x-y|+d(x))^N}{(|x-y|+d(x))^\xa} \gtrsim \frac{|x-y|^{N-\xs}\max \{ |x-y|,d(x),d(y) \}^\xs}{\max \{ |x-y|,d_\Sigma(x),d_\Sigma(y) \}} = \CN_{\xa,\xs}^{-1}(x,y)
\eal
since $d(x)\leq d_\Sigma(x)$.

If, on the other hand, $|x-z|\leq d_\Sigma(x)$, again we have
\bal
&\hspace{-1cm}\CN_{\xa,\xs}^{-1}(x,z)  = \frac{|x-z|^{N-\xs}\max \{ |x-z|,d(x),d(z) \}^\xs}{\max\{ |x-z|,d_\Sigma(x),d_\Sigma(z) \}^\xa} \gtrsim d_\Sigma^{-\xa}(x) |x-z|^{N-\xs} \max \{ |x-y|,d(x) \}^\xs\\
&\gtrsim \frac{|x-y|^{N-\xs}\max \{ |x-y|,d(x),d(y) \}^\xs}{\max \{ |x-y|,d_\Sigma(x),d_\Sigma(y) \}^\xa} = \CN_{\xa,\xs}^{-1}(x,y)
\eal
and the same can be shown to be true in the case $2|x-z|\leq |x-y|$ by a symmetric argument, thus proving \eqref{quasi-norm} for the case $\xa>0$.

\medskip

\paragraph{\textbf{Case 2}} Let $\xa\leq 0$. Since $d_\Sigma(x)\leq |x-y|+d_\Sigma(y)$, it follows that
\bal
\max \{ |x-y|,d_\Sigma(x),d_\Sigma(y) \} \leq |x-y|+\min \{ d_\Sigma(x),d_\Sigma(y) \},
\eal
and from this and the triangle inequality $|x-y|\leq |x-z|+|z-y|$ we estimate
\bal
\frac{|x-y|^{N-\xs}}{\max \{ |x-y|,d_\Sigma(x),d_\Sigma(y) \}^{\xa}} \lesssim \frac{|x-z|^{N-\xs}}{\max \{ |x-z|,d_\Sigma(x),d_\Sigma(z) \}^{\xa}} + \frac{|z-y|^{N-\xs}}{\max \{ |z-y|,d_\Sigma(z),d_\Sigma(y) \}^{\xa}}.
\eal
Now since $d(x)\leq |x-y|+d(y)$, we see in a similar manner that
\bal
\max \{ |x-y|,d(x),d(y) \} \leq |x-y|+\min \{ d(x),d(y) \},
\eal
therefore
\bal
&\hspace{-1cm} \CN_{\xa,\xs}^{-1}(x,y) = \frac{|x-y|^{N-\xs}\max \{ |x-y|,d(x),d(y) \}^\xs}{\max \{ |x-y|,d_\Sigma(x),d_\Sigma(y) \}^\xa}\\
&\lesssim \frac{|x-y|^N}{\max \{ |x-y|,d_\Sigma(x),d_\Sigma(y) \}^{\xa}}+\frac{\min \{ d(x),d(y) \}^\xs |x-y|^{N-\xs}}{\max \{ |x-y|,d_\Sigma(x),d_\Sigma(y) \}^\xs},
\eal
and \eqref{quasi-norm} follows from this and the previous estimate. This completes the proof.
\end{proof}

Let us recall the following estimate for Euclidean balls.

\begin{lemma}[{\cite[Lemma 6.4]{GN2}}]\label{eucl-ball-estimate}
Let $b>0$, $b+\theta>k-N$ and $d\omega = d^b(x)d_\Sigma^\theta(x)\chi_\Omega(x) dx$. Then
\bal
\omega(B(x,s)) \approx \max \{ d(x),s \}^b \max \{ d_\Sigma(x),s \}^\theta s^N \quad \forall x\in \Omega,\ 0<s\leq 4\,\mathrm{diam}(\Omega).
\eal
\end{lemma}

\begin{lemma}
Let $\xa<N$, $b>0$, $\theta > \max \{ k-N-b,-b-\xa \}$ and $d\omega = d^b(x)d_\Sigma^\theta(x)\chi_\Omega(x) dx$. Then \eqref{abstract-condition-1} holds.
\end{lemma}

\begin{proof}
We proceed as in \cite[Lemma 6.5]{GN2}, with minor modifications. Note that if $s\geq (4\diam(\xO))^{N-\xa}$ then $\xo(\mathfrak{B}(x,s))=\xo(\overline{\xO})<\infty$, where $\mathfrak{B}(x,s):= \{ y\in \Omega: \CN_{\xa,\xs}^{-1}(y,x)<s \}$. Let $M=(4\diam\xO)^{N-\xa}.$ We will show that
\ba\label{quasi-ball-estimate-array}
\gw\left(\mathfrak{B}(x,s)\right)\approx\left\{
\BAL
&d(x)^{b-\frac{\xs N}{N-\xs}} d_\xS(x)^{\theta+\frac{\xa N}{N-\xs}} s^\frac{N}{N-\xs}&&\quad \text{if}\; s\leq d(x)^N d_\xS(x)^{-\xa},\\
&s^\frac{b+N}{N}d_\xS(x)^{\theta+\xa\frac{b+N}{N}}&&\quad \text{if}\;d(x)^N d_\xS(x)^{-\xa}<s\leq d_\xS(x)^{N-\xa},\\
&s^\frac{b+\theta+N}{N-\xa}&&\quad \text{if}\; d_\xS(x)^{N-\xa}\leq s\leq M,\\
&M^\frac{b+\theta+N}{N-\xa}&&\quad \text{if}\;  M\leq s,
\EAL\right.
\ea
which, through straightforward calculations, implies
\ba\label{quasi-ball-integral-estimate-array}
\int_0^{s}\frac{\gw\left(\mathfrak{B}(x,t)\right)}{t^2} \dd t\approx\left\{
\BAL
&d(x)^{b-\frac{\xs N}{N-\xs}}d_\xS(x)^{\theta+\frac{\xa N}{N-\xs}}s^\frac{\xs}{N-\xs} &&\quad \text{if}\; s\leq d(x)^N d_\xS(x)^{-\xa}, \\
&s^\frac{b}{N}d_\xS(x)^{\theta+\xa\frac{b+N}{N}}&&\quad \text{if}\;d(x)^Nd_\xS(x)^{-\xa}<s\leq d_\xS(x)^{N-\xa},\\
&s^\frac{b+\theta+\xa}{N-\xa}&&\quad \text{if}\; d_\xS(x)^{N-\xa}\leq s\leq M,\\
&M^\frac{b+\theta+\xa}{N-\xa}&&\quad \text{if}\;  M\leq s,
\EAL\right.
\ea
since $b>0$ and $b+\theta+\xa>0.$ In particular, we obtain \eqref{abstract-condition-1}. In order to demonstrate \eqref{quasi-ball-estimate-array}, we will consider three cases; the fourth he have already considered in the beginning.

\medskip

\noindent \textbf{Case 1:} $s\leq d(x)^Nd_\xS(x)^{-\xa}$.

a) Let $y\in\mathfrak{B}(x,s)$ be such that $d_\xS(x)\leq |x-y|$, and consequently $d(x)\leq|x-y|$. Then
\bal
\CN_{\xa,\xs}^{-1}(y,x)\approx|x-y|^{N-\xa},
\eal
thus if $|x-y|^{N-\xa}\lesssim s\leq d(x)^N d_\xS(x)^{-\xa}\leq d_\xS(x)^{N-\xa}$ then $d(x)\approx d_\xS(x)\approx |x-y|.$ Hence, there exist constants $C_1,C_2>0$ depending on $\xa,N$ such that
\ba\label{case1a}\BAL
& \left\{y\in\xO:|x-y|\leq C_1(d_\xS(x)^\xa d(x)^{-\xs}s)^\frac{1}{N-\xs},\;d_\xS(x)\leq |x-y| \right\}\\
&\quad\subset \left\{y\in\xO:\CN_{\xa,\xs}^{-1}(y,x)\leq s,\;d_\xS(x)\leq |x-y| \right\}\\
&\qquad\subset \left\{y\in\xO:|x-y|\leq C_2(d_\xS(x)^\xa d(x)^{-\xs} s)^\frac{1}{N-\xs},\;d_\xS(x)\leq |x-y| \right\}.
\EAL
\ea

\medskip

b) Let $y\in\mathfrak{B}(x,s)$ be such that $d(x)\leq|x-y|$ and $d_\xS(x)> |x-y|.$ Then
\bal
\CN_{\xa,\xs}^{-1}(y,x)\approx|x-y|^{N}d_\xS(x)^{-\xa},
\eal
so if $|x-y|^{N}d_\xS(x)^{-\xa} \lesssim s,$ then $|x-y|^{N}\lesssim s d_\xS(x)^{\xa}\leq d(x)^N$. It follows that  $d(x)\approx |x-y|.$ Hence, there exist constants $C_1,C_2>0$ depending on $\xa,N$ such that
\bal\BAL
& \left\{y\in\xO:|x-y|\leq C_1(d_\xS(x)^\xa d(x)^{-\xs}s)^\frac{1}{N-\xs},\;d(x)\leq|x-y|,\;d_\xS(x)> |x-y| \right\}\\
&\quad\subset \left\{y\in\xO:\CN_{\xa,\xs}^{-1}(y,x)\leq s,\;d(x)\leq|x-y|,\;d_\xS(x)> |x-y| \right\}\\
&\qquad\subset \left\{y\in\xO:|x-y|\leq C_2(d_\xS(x)^\xa d(x)^{-\xs} s)^\frac{1}{N-\xs},\;d(x)\leq|x-y|,\;d_\xS(x)> |x-y| \right\}.
\EAL
\eal

\medskip

c)  Let $y\in\mathfrak{B}(x,s)$ be such that $d(x)>|x-y|$, and consequently $d_\xS(x) >|x-y|$. Then
\bal
\CN_{\xa,\xs}^{-1}(y,x)\approx|x-y|^{N-\xs}d(x)^\xs d_\xS(x)^{-\xa}.
\eal
Hence, there exist constants $C_1,C_2>0$ depending on $\xa,N$ such that
\ba\label{case1c}\BAL
& \left\{x\in\xO:|x-y|\leq C_1(d_\xS(x)^\xa d(x)^{-\xs} s)^\frac{1}{N-\xs},\;d(x)>|x-y| \right\}\\
&\quad\subset \left\{x\in\xO:\CN_{\xa,\xs}^{-1}(y,x)\leq s,\;d(x)>|x-y| \right\}\\
&\qquad\subset \left\{x\in\xO:|x-y|\leq C_2(d_\xS(x)^\xa d(x)^{-\xs} s)^\frac{1}{N-\xs},\;d(x)>|x-y|\right\}.
\EAL
\ea

By \eqref{case1a}--\eqref{case1c} and Lemma \ref{eucl-ball-estimate}, we infer
\bal
\xo(\mathfrak{B}(x,s))\approx\xo(B(x,s_1))\approx d(x)^{b-\frac{\xs N}{N-\xs}} d_\xS(x)^{\theta+\frac{\xa N}{N-\xs}} s^\frac{N}{N-\xs},
\eal
where $s_1=(d_\xS(x)^\xa d(x)^{-\xs}s)^\frac{1}{N-\xs}$.

\medskip

\noindent \textbf{Case 2:} $ d(x)^N d_\xS(x)^{-\xa}<s\leq d_\xS(x)^{N-\xa}.$

a) Let $y\in\mathfrak{B}(x,s)$ be such that $d_\xS(x)\leq |x-y|$, and consequently $d(x)\leq |x-y|$. Then
\bal
\CN_{\xa,\xs}^{-1}(y,x)\approx|x-y|^{N-\xa}.
\eal
So if $|x-y|^{N-\xa}\lesssim s\leq d_\xS(x)^{N-\xa}$, then $d_\xS(x)\approx |x-y|$, and there exist constants $C_1,C_2>0$ which depend on $\xa,N$ such that
\ba\label{case2a}\BAL
&\left\{y\in\xO:|x-y|\leq C_1(d_\xS(x)^\xa s)^\frac{1}{N},\;d_\xS(x)\leq |x-y| \right\}\\
&\quad\subset \left\{y\in\xO:\CN_{\xa,\xs}^{-1}(y,x)\leq s,\;d_\xS(x)\leq |x-y| \right\}\\
&\qquad\subset \left\{y\in\xO:|x-y|\leq C_2(d_\xS(x)^\xa s)^\frac{1}{N},\;d_\xS(x)\leq |x-y| \right\}.
\EAL
\ea

\medskip

b) Let $y\in\mathfrak{B}(x,s)$ be such that $d(x)\leq|x-y|$ and $d_\xS(x)> |x-y|.$ Then
\bal
\CN_{\xa,\xs}^{-1}(y,x)\approx|x-y|^{N}d_\xS(x)^{-\xa}.
\eal
It follows that there exist constants $C_1,C_2>0$ depending on $\xa,N$ such that
\bal\BAL
&\left\{y\in\xO:|x-y|\leq C_1(d_\xS(x)^\xa s)^\frac{1}{N},\;d(x)\leq|x-y|,\;d_\xS(x)> |x-y|\right \}\\
&\quad\subset \left\{y\in\xO:\CN_{\xa,\xs}^{-1}(y,x)\leq s,\;d(x)\leq|x-y|,\;d_\xS(x)> |x-y|\right \}\\
&\qquad\subset \left\{y\in\xO:|x-y|\leq C_2(d_\xS(x)^\xa s)^\frac{1}{N},\;d(x)\leq|x-y|,\;d_\xS(x)> |x-y| \right\}.
\EAL
\eal

\medskip

c)  Let $y\in\mathfrak{B}(x,s)$ be such that $d(x)>|x-y|$, and consequently $d_\xS(x)> |x-y|$. Then,
\bal
\CN_{\xa,\xs}^{-1}(y,x)\approx|x-y|^{N-\xs}d(x)^\xs d_\xS(x)^{-\xa},
\eal
\bal
|x-y|^{N-\xs}d(x)^\xs d_\xS(x)^{-\xa}\geq |x-y|^{N}d_\xS(x)^{-\xa},
\eal
and
\bal
|x-y|\leq(d_\xS(x)^\xa s)^\frac{1}{N}=(d_\xS(x)^\xa s)^\frac{1}{N-\xs}(d_\xS(x)^\xa s)^{-\frac{\xs}{N(N-\xs)}}
\leq(d_\xS(x)^\xa d(x)^{-\xs} s)^\frac{1}{N-\xs},
\eal
since $d(x)^N d_\xS(x)^{-\xa}<s$. Hence, there exist constants $C_1,C_2>0$ depending on $\xa,N$ such that
\ba\label{case2c}\BAL
&\left\{y\in\xO:|x-y|\leq C_1(d_\xS(x)^\xa s)^\frac{1}{N},\;d(x)>|x-y| \right \}\\
&\quad \subset \left\{y\in\xO:\CN_{\xa,\xs}^{-1}(y,x)\leq s,\;d(x)>|x-y| \right\}\\
&\qquad \subset \left\{y\in\xO:|x-y|\leq C_2(d_\xS(x)^\xa s)^\frac{1}{N},\;d(x)>|x-y| \right\}.
\EAL
\ea

By \eqref{case2a}-\eqref{case2c} and Lemma \ref{eucl-ball-estimate}, we conclude
\bal
\xo(\mathfrak{B}(x,s))\approx\xo(B(x,s_2))\approx s^\frac{b+N}{N}d_\xS(x)^{\theta+\xa\frac{b+N}{N}},
\eal
where $s_2=(d_\xS(x)^{\xa}s)^\frac{1}{N}.$

\medskip

\noindent \textbf{Case 3:} $  d_\xS(x)^{N-\xa}<s\leq (4\,\diam(\xO))^{N-\xa}.$

a) Let $y\in\mathfrak{B}(x,s)$ be such that $d_\xS(x)\leq|x-y|$, and consequently $d(x)\leq |x-y|.$ Then
\bal
\CN_{\xa,\xs}^{-1}(y,x)\approx|x-y|^{N-\xa}.
\eal
Hence, there exist constants $C_1,C_2>0$ which depend on $\xa,N$ such that
\ba\label{case3a}\BAL
&\left\{y\in\xO:|x-y|\leq C_1s^\frac{1}{N-\xa},\;d_\xS(x)\leq |x-y| \right\}\\
&\quad\subset \left\{y\in\xO:\CN_{\xa,\xs}^{-1}(y,x)\leq s,\;d_\xS(x)\leq |x-y| \right\}\\
&\qquad\subset \left\{y\in\xO:|x-y|\leq C_2s^\frac{1}{N-\xa},\;d_\xS(x)\leq |x-y| \right\}.
\EAL
\ea

\medskip

b) Let $y\in\mathfrak{B}(x,s)$ be such that $d(x)\leq|x-y|$ and $d_\xS(x)> |x-y|.$ Then
\bal
\CN_{\xa,\xs}^{-1}(y,x)\approx|x-y|^{N}d_\xS(x)^{-\xa}.
\eal

On the one hand, if $\xa>0,$ we have
\bal
|x-y|^{N}d_\xS(x)^{-\xa}\leq|x-y|^{N-\xa}
\eal
and
\bal
|x-y|\leq(d_\xS(x)^\xa s)^\frac{1}{N}=s^\frac{1}{N-\xa} s^{-\frac{\xa}{N(N-\xa)}} d_\xS(x)^{\frac{\xa}{N}}
\leq s^\frac{1}{N-\xa},
\eal
since $d_\xS(x)^{N-\xa}<s$.

On the other hand, if $\xa\leq 0$ then
\bal
|x-y|^{N}d_\xS(x)^{-\xa}\geq|x-y|^{N-\xa}
\eal
and
\bal
|x-y|\leq s^\frac{1}{N-\xa}=s^\frac{1}{N} s^{\frac{\xa}{N(N-\xa)}}\leq (d_\xS(x)^\xa s)^\frac{1}{N},
\eal
since $d_\xS(x)^{N-\xa}<s$.

Hence, there exist constants $C_1,C_2>0$ depending on $\xa,N$ such that
\bal\BAL
&\left \{y\in\xO:|x-y|\leq C_1s^\frac{1}{N-\xa},\;d(x)\leq|x-y|,\;d_\xS(x)> |x-y| \right \}\\
&\quad\subset \left\{y\in\xO:\CN_{\xa,\xs}^{-1}(y,x)\leq s,\;d(x)\leq|x-y|,\;d_\xS(x)> |x-y| \right \}\\
&\qquad\subset \left \{y\in\xO:|x-y|\leq C_2s^\frac{1}{N-\xa},\;d(x)\leq|x-y|,\;d_\xS(x)> |x-y| \right \}.
\EAL
\eal

\medskip

c)  Let $y\in\mathfrak{B}(x,s)$ be such that $d(x)>|x-y|.$ Then $d_\xS(x)> |x-y|$ and
\bal
\CN_{\xa,\xs}^{-1}(y,x)\approx|x-y|^{N-\xs}d(x)^\xs d_\xS(x)^{-\xa}.
\eal

By \eqref{case2c}, we may infer the existence of positive constants $C_1,C_2>0$ depending on $\xa,N$ such that
\bal
&\left \{y\in\xO:|x-y|\leq C_1(d_\xS(x)^\xa s)^\frac{1}{N},\;d(x)>|x-y| \right \}\\
&\quad\subset \left\{y\in\xO:\CN_{\xa,\xs}^{-1}(y,x)\geq s,\;d(x)>|x-y| \right\}\\
&\qquad\subset \left\{y\in\xO:|x-y|\leq C_2(d_\xS(x)^\xa s)^\frac{1}{N},\;d(x)>|x-y| \right\}.
\eal
This and the fact that $s>d_\xS^{N-\xa}(x)$ imply the existence of two positive constants $\tilde C_1,\tilde C_2>0$ depending only on $\xa,N$ such that
\ba\label{case3c}\BAL
&\left\{y\in\xO:|x-y|\leq \tilde C_1s^\frac{1}{N-\xa},\;d(x)>|x-y| \right \}\\
&\quad\subset \left \{y\in\xO:\CN_{\xa,\xs}^{-1}(y,x)\leq s,\;d(x)>|x-y| \right \}\\
&\qquad\subset \left\{y\in\xO:|x-y|\leq \tilde C_2s^\frac{1}{N-\xa},\;d(x)>|x-y| \right\}.
\EAL
\ea

Finally, by \eqref{case3a}-\eqref{case3c} and Lemma \ref{eucl-ball-estimate}, we obtain
\bal
\xo(\mathfrak{B}(x,s))\approx\xo(B(x,s_3))\approx s^\frac{b+\theta+N}{N-\xa},
\eal
where $s_3=s^\frac{1}{N-\xa}.$ This completes the proof.
\end{proof}

\begin{lemma}
Let $\xa<N$, $b>0$, $\theta> \max \{ k-N-b,-b-\xa \}$ and $d\omega = d^b(x)d_\Sigma^\theta(x)\chi_\Omega(x) dx$. Then \eqref{abstract-condition-2} holds.
\end{lemma}

\begin{proof}
The proof is identical to that of \cite[Lemma 6.6]{GN2}, so we omit it. We simply remark that estimate \eqref{quasi-ball-integral-estimate-array} is used, but the $\xs$ dependence becomes irrelevant.
\end{proof}

Combining these results we apply Proposition \ref{abstract-proposition} with $\BBJ=\BBN_{\xa,\xs}$, $d\xo=d^b(x)d_\Sigma^\theta(x)dx$ and $\lambda=\nu$ to obtain:

\begin{lemma}\label{Concrete-conditions}
Let $\xa<N$, $0<\xs<N$, $b>0$ and $\theta>\max\{ k-N-b,-b-\xa \}$. Moreover, assume $p,q\geq 0$ and $p+q>1$, and let $\nu\in \GTM_+(\partial \Omega)$. Then the following conditions are equivalent:
\begin{enumerate}
\item For $l>0$ small, the following equation possesses a positive solution:
\be\label{concrete-condition-1}
v= \BBN_{\xa,\xs}[|v|^{p+q}d^bd_\Sigma^\theta] + l \BBN_{\xa,\xs}[\nu]
\ee
\item For each $B\in \CB(\overline{\Omega})$, there holds
\be\label{concrete-condition-2}
\int_B \BBN_{\xa,\xs}[\chi_B \nu]^{p+q} d^b d_\Sigma^\theta\,dx \leq C\nu(B).
\ee
\item The following inequality holds:
\be\label{concrete-condition-3}
\BBN_{\xa,\xs}[\BBN_{\xa,\xs}[\nu]^{p+q}d^b d_\Sigma^\theta] \leq C \BBN_{\xa,\xs}[\nu].
\ee
\item For each $B\in \CB(\overline{\Omega})$, there holds
\be\label{concrete-condition-4}
\nu(B) \leq C \mathrm{Cap}_{\BBN_{\xa,\xs},d^b d_\Sigma^\theta}^{(p+q)'}(B).
\ee
\end{enumerate}
\end{lemma}

We are finally able to prove:

\begin{theorem}\label{Existence-capacity}
Let $p,q\geq 0$ with $p+q>1$ and $$\xa_-<\frac{p+1}{p+q-1}.$$ Let $\nu\in\GTM_+(\partial \Omega)$ be such that $\| \nu \|_{\GTM(\partial\Omega)}=1$ and
\bal
\nu(B) \leq C\, \mathrm{Cap}_{\BBN_{2\xa_-,1},d^{p+1}d_\Sigma^{-(p+q+1)\xa_-}}^{(p+q)'}(B) \quad \forall B\in \CB(\overline{\Omega}).
\eal
Then there exists $\vgr>0$ such that the boundary value problem
\be\label{power-BVP} \left\{ \BAL
- L_\gm u&= |u|^p|\nabla u|^q &\text{ in } \Omega\\
\tr(u)&=\vgr \nu &
\EAL \right.
\ee
possesses a non-negative weak solution.
\end{theorem}

\begin{proof}
We proceed as in the proof of \cite[Theorem 1.11]{GN-gradient}, with modifications to account for the different distance functions involved. The assumption of the theorem is the same as \eqref{concrete-condition-4} where $b=p+1$, $\theta=-(p+q+1)\xa_-$, so all conditions \eqref{concrete-condition-1}--\eqref{concrete-condition-4} apply equivalently. We define the operator
\bal
\BBA[u]:= \BBG_\mu[|u|^p|\nabla u|^q]+\BBK_\mu[\vgr \nu]
\eal
with the intention of applying the Schauder fixed point theorem. As we will see shortly, an appropriate domain for $\BBA$ is the space $$V:= \{ v\in W^{1,1}_{\loc}(\Omega): v\in L^{p+q}(\Omega;d^{1-q}d_\Sigma^{-\xa_-}) \text{ and } \nabla v \in L^{p+q}(\Omega;d^{1+p}d_\Sigma^{-\xa_-}) \}$$ with the obvious choice of topology. Note that H\"older's inequality ensures that $|v|^p|\nabla v|^q\in L^1(\Omega;\ei)$ whenever $v\in V$ so that the operator is well defined.

First we estimate
\bal
G_\mu(x,y) &\approx d(x)d_\Sigma^{-\xa_-}(x)d(y)d_\Sigma^{-\xa_-}(y) \CN_{2\xa_-,2}(x,y)\\
	&\quad\lesssim d(x)d_\Sigma^{-\xa_-}(x)d(y)d_\Sigma^{-\xa_-}(y) \CN_{2\xa_-,1}(x,y)
\eal
and similarly
\bal
|\nabla_x G_\mu(x,y)| &\lesssim d_\Sigma^{-\xa_-}(x)d(y)d_\Sigma^{-\xa_-}(y) \CN_{2\xa_-,1}(x,y) \\
K_\mu(x,\xi) &\approx d(x)d_\Sigma^{-\xa_-}(x) \CN_{2\xa_-,1}(x,\xi) \\
|\nabla_x K_\mu(x,\xi)| &\lesssim d_\Sigma^{-\xa_-}(x) \CN_{2\xa_-,1}(x,\xi),
\eal
and from these we obtain
\bal
|\BBA[u]| &\leq C_1 (dd_\Sigma^{-\xa_-}\BBN_{2\xa_-,1}[dd_\Sigma^{-\xa_-}|u|^p|\nabla u|^q]+dd_\Sigma^{-\xa_-}\BBN_{2\xa_-,1}[\vgr\nu]),\\
|\nabla \BBA[u]| &\leq C_1 (d_\Sigma^{-\xa_-}\BBN_{2\xa_-,1}[dd_\Sigma^{-\xa_-}|u|^p|\nabla u|^q]+d_\Sigma^{-\xa_-}\BBN_{2\xa_-,1}[\vgr\nu])
\eal
for some uniform constant $C_1>0$. Now set $$E:= \{ u\in W^{1,1}_\loc(\Omega): |u|\leq 2C_1 dd_\Sigma^{-\xa_-}\BBN_{2\xa_-,1}[\vgr\nu] \text{ and } |\nabla u|\leq 2C_1 d_\Sigma^{-\xa_-}\BBN_{2\xa_-,1}[\vgr\nu] \},$$ and observe that the equivalent condition \eqref{concrete-condition-2} implies that $E\subset V$. Moreover, \eqref{concrete-condition-3} implies that there exists $\vgr_0=\vgr_0(p,q,C_1,C)>0$ such that
\bal
\vgr\in (0,\vgr_0) \Rightarrow \BBA[E]\subset E.
\eal
$E$ is easily seen to be closed and convex in $V$, and one can show that $\BBA:E\rightarrow \BBA[E]$ is continuous and compact. Thus the Schauder fixed point theorem applies, yielding a solution $u\in E$ of $u=\BBA[u]$, which is consequently a weak solution of \eqref{power-BVP}. Non-negativity can be obtained by restriction to the positive cone of $E$.
\end{proof}

In view of the weak Lebesgue estimates for the Martin operator, we also easily obtain the following.

\begin{proof}[\textbf{Proof of Theorem \ref{powersubcritical}}]
In view of the proof of Theorem \ref{Existence-capacity}, it suffices to demonstrate condition \eqref{abstract-condition-2} with $b=p+1$, $\theta=-(p+q+1)\xa_-$, $\xa=2\xa_-$, $\sigma=1$. In that proof we have seen that
\bal
K_\mu(x,\xi)\approx d(x)d_\Sigma^{-\xa_-}(x)\CN_{2\xa_-,1}(x,\xi),
\eal
so it follows that for all $B\in\CB(\overline{\Omega})$ there holds
\bal
&\int_B \BBN_{2\xa_-,1}[\chi_B \nu]^{p+q}d^{p+1}d_\Sigma^{-(p+q+1)\xa_-}\,dx \\
&\hspace{1cm} \approx \int_B \BBK_\mu[\chi_B \nu]^{p+q}d^{1-q}d_\Sigma^{-\xa_-}\,dx \\
&\hspace{2cm} \lesssim \int_\Omega (\BBK_\mu[\chi_B \nu] d^{-\frac{q}{p+q}})^{p+q} \ei\,dx \\
&\hspace{3cm} = \| \BBK_{\mu,\xg}[\chi_B\nu] \|_{L^{p+q}(\Omega;\ei)}
\eal
where $\xg=q/(p+q)$. Now observe that the two alternative conditions in the statement of the theorem are chosen so that $\| \BBK_{\mu,\xg}[\chi_B\nu] \|_{L^{p+q}(\Omega;\ei)}\leq C \nu(B)$ in view of Theorem \ref{martin-est-p1-p2}. This completes the proof.

\end{proof}

In the sequel we will need the following property, due to \cite[Theorem 1.5.2]{Ad}, which states that
\be\label{capacity-alternative}
(\mathrm{Cap}_{\BBN_{\xa,\xs},d^bd_\Sigma^\theta}^s(E))^{1/s} = \sup\{\omega(E): \omega\in\GTM_+(E),\, \| \BBN_{\xa,\xs}[\omega]\|_{L^{s'}(\Omega,d^bd_\Sigma^\theta)} \leq 1\}.
\ee

For the following results we will need the notion of \textit{Bessel capacity}, which we briefly recall. Let $\langle \xi \rangle := (1+|\xi|^2)^{1/2}$ for $\xi\in\R^d$, and for $\xa\in\R$ let
\bal
B_{d,\xa}(x):= \CF^{-1}(\langle \cdot \rangle^{-\xa})(x)=\int_{\R^d} e^{2\pi i x\xi} \langle \xi \rangle^{-\xa}\, d\xi.
\eal
The Bessel kernel and operator of order $\xa$ are then defined as follows
\bal
\CB_{d,\xa}(x,y):= B_{d,\xa}(x-y),\quad \BBB_{d,\xa}[\tau](x):=\int_{\R^d}\CB_{d,\xa}(x,y)\,d\tau(y)\quad \forall \tau\in \GTM(\R^d).
\eal
Moreover, define the space $L^\xk_\xa(\R^d):= \BBB_{d,\xa}[L^\xk(\R^d)]$ with the norm
\bal
\| \BBB_{d,\xa}[f] \|_{L^\xk_\xa(\R^d)}:= \| f \|_{L^\xk(\R^d)}.
\eal

Now let $1<\xk<\infty$ and $E\subset \R^d$, and set
\bal
\CS_\xa^\xk(E):= \{ f\in L^\xk(\R^d): f\geq 0 \text{ and } \BBB_{d,\xa}[f]\geq \chi_E \}.
\eal
The Bessel $\xk$-capacity of order $\xa$ of $E$ is the defined as follows
\bal
\mathrm{Cap}_{\CB_{d,\xa}}^\xk(E):= \inf\Big\{ \int_{\R^d} f^\xk\,dx : f\in \CS_\xa^\xk(E) \Big\}.
\eal

This is not immediately applicable in our context, but it will be after flattening the submanifold we want to study. Recall that if $\Gamma$ is a $C^2$ $k$-submanifold of $\partial\Omega$ without boundary, then there exist $O_1,\ldots,O_m \subset \R^N$, diffeomorphisms $T_i: O_i \rightarrow B^k(0,1)\times B^{N-k-1}(0,1)\times (-1,1)$ and compact sets $K_1,\ldots,K_m$ such that
\begin{enumerate}
\item $K_i\subset O_i$ for all $i=1,\ldots,m$ and $\bigcup_{i=1}^m K_i=\Gamma$,
\item $T_i(O_i\cap \Gamma) = B^k(0,1)\times \{ 0_{\R^{N-k}} \}$ and $T_i(O_i \cap \Omega)=B^k(0,1)\times B^{N-k-1}(0,1)\times (0,1)$,
\item For any $x\in O_i\cap \Omega$ there exists $y\in O_i\cap \Gamma$ such that $d_\Gamma(x)=|x-y|$.
\end{enumerate}
With this settled, let us define the final notion of $\Gamma$-capacity which we will use in the sequel. Set
\ba\label{besselcapacity}
\mathrm{Cap}^{\Gamma,s}_{\xa}(E):= \sum_{i=1}^m \mathrm{Cap}^s_{\CB_{k,\xa}}(\pi_k \circ T_i(E\cap K_i)) \quad \forall E\subset \Gamma \text{ compact},
\ea
where $\pi_k: \R^k \times \R^{N-k}\rightarrow \R^k$ denotes the projection on the first $k$ components. For later convenience we also set $\tilde{T}_i = \pi_k\circ T_i$. We remark that the definition is independent of the $O_i$. We will be particularly interested in the cases $\Gamma=\Sigma$ and $\Gamma=\partial\Omega$.

\begin{proof}[\textbf{Proof of Theorem \ref{intro-Bessel-Sigma}}]
In view of Lemma \ref{Concrete-conditions} and Theorem \ref{Existence-capacity}, it suffices to show that
\bal
\mathrm{Cap}^{(p+q)'}_{\CN_{2\xa_-,1},d^b d_\Sigma^\theta}(E) \approx \mathrm{Cap}^{\Sigma,(p+q)'}_{\vartheta}(E),
\eal
where $b=p+1$, $\theta=-(p+q+1)\xa_-$ and $\vartheta$ as in the statement of the theorem. By definition we have that
\bal
\mathrm{Cap}^{\Sigma,(p+q)'}_{\vartheta}(E):= \sum_{i=1}^m \mathrm{Cap}^{(p+q)'}_{\CB_{k,\vartheta}}(\pi_k \circ T_i(E\cap K_i)),
\eal
and it also follows easily that
\bal
\mathrm{Cap}^{(p+q)'}_{\CN_{2\xa_-,1},d^b d_\Sigma^\theta}(E) \approx \sum_{i=1}^m \mathrm{Cap}^{(p+q)'}_{\CN_{2\xa_-,1},d^b d_\Sigma^\theta}(E\cap K_i),
\eal
so if we are able to show that
\bal
\mathrm{Cap}^{(p+q)'}_{\CN_{2\xa_-,1},d^b d_\Sigma^\theta}(E\cap K_i) \approx \mathrm{Cap}^{(p+q)'}_{\CB_{k,\vartheta}}(\tilde{T}_i(E\cap K_i))
\eal
we are done.

We intend to use equality \eqref{capacity-alternative}. Let $\xl\in\GTM_+(\Sigma)$ be such that $\BBN_{2\xa_-,1}[\lambda] \in L^{p+q}(\Omega,d^bd_\Sigma^\theta)$ and set $d\lambda_{K_i}:= \chi_{K_i}d\lambda$. On the one hand, since for $x\in\Omega$ and $\xi\in\Sigma$ we have $\CN_{\xa,\xs}(x,\xi) = |x-\xi|^{-N+\xa}$ it follows that
\bal
\int_{O_i} \BBN_{2\xa_-,1}[\lambda_{K_i}]^{p+q}d^bd_\Sigma^\theta\, dx \gtrsim \lambda(K_i)^{p+q} \int_{O_i} d^bd_\Sigma^\theta\,dx \approx \lambda(K_i)^{p+q}.
\eal
On the other hand, arguing in a similar way reveals that
\bal
\int_{\Omega\setminus O_i} \BBN_{2\xa_-,1}[\lambda_{K_i}]^{p+q}d^bd_\Sigma^\theta\, dx \lesssim \lambda(K_i)^{p+q} \int_{\Omega} d^bd_\Sigma^\theta\,dx \approx \lambda(K_i)^{p+q},
\eal
so overall we have
\bal
\int_{\Omega} \BBN_{2\xa_-,1}[\lambda_{K_i}]^{p+q}d^bd_\Sigma^\theta\, dx \approx \int_{O_i} \BBN_{2\xa_-,1}[\lambda_{K_i}]^{p+q}d^bd_\Sigma^\theta\, dx.
\eal
We now pass to the flattening coordinates $T_i(x)=(\psi'(x),\tilde{\psi}(x),\psi_N(x))$, where $d\approx\psi_n$, $d_\Sigma\approx(\psi_N^2+|\tilde{\psi}|^2)^{1/2}\approx \psi_N+ |\tilde{\psi}|$ and
\bal
\CN_{\xa,\xs}(x,\xi) = |x-\xi|^{-N+\xa} \approx (\psi_N+ |\tilde{\psi}|+|\psi'-\xi'|)^{-N+\xa}.
\eal
Now set $\overline{\lambda}_i \in \GTM_+(B^k(0,1))$ with $\overline{\lambda}_i(E):= \lambda(\tilde{T}_i^{-1}(E\cap K_i))$. Then:
\bal
\int_{O_i} \BBN_{2\xa_-,1}[\lambda_{K_i}]^{p+q}d^bd_\Sigma^\theta\, dx &\approx \int_{B^k(0,1)}\int_{B^{N-k-1}(0,1)}\int_0^1 \psi_N^{p+1}(\psi_N+|\tilde{\psi}|)^{-(p+q+1)\xa_-}\\
&\hspace{1cm} \times\bigg( \int_{B^k(0,1)} (\psi_N+ |\tilde{\psi}|+|\psi'-\xi'|)^{-N+2\xa_-}\, d\overline{\lambda}_i(\xi') \bigg)^{p+q}\, d\psi_N d\tilde{\psi} d\psi' \\
&\approx \int_{B^k(0,1)} \int_0^1 r^{N-k-1+p+1-(p+q+1)\xa_-} \\
&\hspace{1cm} \times\bigg( \int_{B^k(0,1)} (r+|\psi'-\xi'|)^{-N+2\xa_-}\, d\overline{\lambda}_i(\xi') \bigg)^{p+q}\, drd\psi'\\
&\approx \int_{\R^k} \BBB_{k,\gv}[\overline{\lambda}_i]^{p+q}(x')\,dx',
\eal
where the last estimate is derived in detail in the next lemma. Thus we have proven that
\bal
\| \BBN_{2\xa_-,1}[\lambda_{K_i}] \|_{L^{p+q}(\Omega,d^bd_\Sigma^\theta)} \approx \| \BBB_{k,\gv}[\overline{\lambda}_i] \|_{L^{p+q}(\R^k)},
\eal
which implies the conclusion in view of \eqref{capacity-alternative}.
\end{proof}

To complete the proof we need the following lemma.

\begin{lemma}
Let $p,q\geq 0$ with $p+q>1$ and $\nu\in \GTM_+(\R^k)$ with $\supp \nu \in B^k(0,R/2)$ for some $R>0$ and let 
\bal
\gv:= k-N+2\xa_-+\frac{N-k+p+1-(p+q+1)\xa_-}{p+q}.
\eal
Assume that $0<\gv<k$ and that $N-k>(p+q+1)\xa_--p-1$. For $x\in \R^{k+1}$ we write $x=(x_1,x')$. Then
\bal
&\hspace{-1cm} \int_{B^k(0,1)} \int_0^1 x_1^{N-k-1+p+1-(p+q+1)\xa_-} \\
&\times\bigg( \int_{B^k(0,1)} (x_1+|x'-y'|)^{-N+2\xa_-}\, d\overline{\lambda}_i(y') \bigg)^{p+q}\, drd\psi' \\
&\hspace{1cm} \approx \int_{\R^k} \BBB_{k,\gv}[\nu]^{p+q}(x')\,dx',
\eal
where the implicit constant depends on $R,N,k,\mu,p$.
\end{lemma}

\begin{proof}
\textbf{Upper bound.} Let $0<x_1<R$ and $|x'|<R$. Arguing as in the proof of \cite[Lemma 3.1.1]{Ad} we see that
\bal
&\hspace{-1cm} \int_{B^k(0,R)} (x_1+|x'-y'|)^{-(N-2\xa_-)}d\nu(y') \\
&\leq \int_{B^k(x',2R)} (x_1+|x'-y'|)^{-(N-2\xa_-)}d\nu(y') \\
&\hspace{1cm}= (N-2\xa_-)  \int_0^{2R}\frac{\nu(B^k(x',r))}{(x_1+r)^{N-2\xa_-}} \frac{dr}{x_1+r} + \frac{\nu(B^k(0,2R))}{(x_1+2R)^{N-2\xa_-}} \\
&\hspace{2cm}\lesssim \int_0^{3R} \frac{\nu(B^k(x',r))}{(x_1+r)^{N-2\xa_-}} \frac{dr}{x_1+r} \\
&\hspace{3cm}\leq \int_{x_1}^{4R} \frac{\nu(B^k(x',r))}{r^{N-2\xa_-}} \frac{dr}{r}.
\eal
Now set $\xb:=(p+q+1)\xa_--p-1$. By the above estimate it follows that
\bal
&\hspace{-1cm}\int_0^R x_1^{N-k-1-\xb} \bigg(\int_{B^k(0,R)} (x_1+|x'-y'|)^{-(N-2\xa_-)}d\nu(y')\bigg)^{p+q}dx_1 \\
&\lesssim \int_0^R x_1^{N-k-1-\xb} \bigg( \int_{x_1}^{4R} \frac{\nu(B^k(x',r))}{r^{N-2\xa_-}} \frac{dr}{r} \bigg)^{p+q} dx_1.
\eal
Now let $0<\xe<N-k-\xb$ (note that this is valid if $\xb<N-k$). By H\"older's inequality and Fubini's theorem we see that
\bal
&\hspace{-1cm} \int_0^R x_1^{N-k-1-\xb} \bigg( \int_{x_1}^{4R} \frac{\nu(B^k(x',r))}{r^{N-2\xa_-}} \frac{dr}{r} \bigg)^{p+q} dx_1 \\
&\leq \int_0^R x_1^{N-k-1-\xb} \bigg(\int_{x_1}^\infty r^{-\xe\frac{(p+q)'}{p+q}} \frac{dr}{r} \bigg)^{\frac{p+q}{(p+q)'}} \int_{x_1}^{4R} \bigg( \frac{\nu(B^k(x',r))}{r^{N-2\xa_--\frac{\xe}{p+q}}} \bigg)^{p+q} \frac{dr}{r} dx_1 \\
&\hspace{1cm} = C(p,q,\xe) \int_0^R x_1^{N-k-1-\xb-\xe} \int_{x_1}^{4R} \bigg( \frac{\nu(B^k(x',r))}{r^{N-2\xa_--\frac{\xe}{p+q}}} \bigg)^{p+q} \frac{dr}{r} dx_1 \\
&\hspace{2cm} \leq C(p,q,\xe,N,k,\xa_-,R) \int_{0}^{4R} \bigg( \frac{\nu(B^k(x',r))}{r^{N-2\xa_--\frac{N-k-\xb}{p+q}}} \bigg)^{p+q} \frac{dr}{r}.
\eal
Note that $\gv$ is actually defined so that
\bal
N-2\xa_--\frac{N-k-\xb}{p+q}=k-\gv.
\eal

Next, we estimate
\bal
&\int_{0}^{4R} \bigg( \frac{\nu(B^k(x',r))}{r^{k-\gv}} \bigg)^{p+q} \frac{dr}{r} \\
&\hspace{1cm}= \sum_{n=0}^\infty \int_{2^{-n+1}R}^{2^{-n+2}R} \bigg( \frac{\nu(B^k(x',r))}{r^{k-\gv}} \bigg)^{p+q} \frac{dr}{r} \\
&\hspace{2cm} \leq \ln 2 \sum_{n=0}^\infty 2^{(p+q)(n-1)(k-\gv)} \bigg( \frac{\nu(B^k(x',2^{-n+2}R))}{R^{k-\gv}} \bigg)^{p+q} \\
&\hspace{3cm} \leq \ln 2 \bigg( \sum_{n=0}^\infty 2^{(n-1)(k-\gv)} \frac{\nu(B^k(x',2^{-n+2}R))}{R^{k-\gv}} \bigg)^{p+q} \\
&\hspace{4cm} \leq 4^{(p+q)(k-\gv)} (\ln 2)^{-(p+q-1)} \bigg( \sum_{n=0}^\infty \int_{2^{-n+2}R}^{2^{-n+3}R} \frac{\nu(B^k(x',r))}{r^{k-\gv}} \frac{dr}{r} \bigg)^{p+q} \\
&\hspace{5cm}\leq 4^{(p+q)(k-\gv)} (\ln 2)^{-(p+q-1)} \bigg( \int_0^{8R} \frac{\nu(B(x',r))}{r^{k-\gv}} \frac{dr}{r} \bigg)^{p+q}.
\eal
Setting
\bal
\BBW_{\gv,8R}[\nu](x') := \int_0^{8R} \frac{\nu(B^k(x',r))}{r^{k-\gv}}\frac{dr}{r},
\eal
we see that
\bal
&\hspace{-1cm}\int_0^R x_1^{N-k-1-\xb} \bigg(\int_{B^k(0,R)} (x_1+|x'-y'|)^{-(N-2\xa_-)}d\nu(y')\bigg)^{p+q}dx_1 \\
&\lesssim \int_{\R^k} \BBW_{\gv,8R}[\nu](x')^{p+q} dx' \\
&\hspace{1cm} \lesssim \int_{\R^k} \BBB_{k,\gv}[\nu](x')^{p+q}dx',
\eal
the last inequality owing to \cite[Lemma 2.3]{BVN} with parameters $\tilde{p}=2$, $\xa= \gv/2$, $\tilde q=s=p+q$, $N=k$. (We denote here by $\tilde{p},\tilde{q}$ the parameters $p,q$ in \cite[Theorem 2.3]{BVN}. In addition, to apply this lemma, we should assume that $0<\theta<k,$ for this reason we obtain the assumptions on $p$ and $q$.)

\medskip

\noindent
\textbf{Lower bound.} Let $0<x_1<R$ and $|x'|<R$. By \cite[Lemma 3.1.1]{Ad} we have that
\bal
&\hspace{-1cm} \int_{B^k(0,R)} (x_1+|x'-y'|)^{-(N-2\xa_-)}d\nu(y') \\
&= (N-2\xa_-) \int_{x_1}^\infty \frac{\nu(B^k(x',r-x_1))}{r^{N-2\xa_-}} \frac{dr}{r} \\
&\hspace{1cm}\geq (N-2\xa_-) \int_{2x_1}^\infty \frac{\nu(B^k(x',r/2))}{r^{N-2\xa_-}} \frac{dr}{r} \\
&\hspace{2cm}\geq C(N,\xa_-) \int_{x_1}^\infty \frac{\nu(B^k(x',r))}{r^{N-2\xa_-}} \frac{dr}{r}.
\eal
Then it follows that
\bal
&\hspace{-1cm} \int_0^R x_1^{N-k-1-\xb} \bigg( \int_{B^k(0,R)} (x_1+|x'-y'|)^{-(N-2\xa_-)} d\nu(y') \bigg)^{p+q}dx_1 \\
&\gtrsim \int_0^R x_1^{N-k-1-\xb} \bigg( \int_{x_1}^\infty \frac{\nu(B^k(x',r))}{r^{N-2\xa_-}} \frac{dr}{r} \bigg)^{p+q}dx_1 \\
&\hspace{1cm}\geq \int_0^R x_1^{N-k-1-\xb} \bigg( \int_{x_1}^{2x_1} \frac{\nu(B^k(x',r))}{r^{N-2\xa_-}} \frac{dr}{r} \bigg)^{p+q}dx_1 \\
&\hspace{2cm}\gtrsim \int_0^R \bigg( \frac{\nu(B^k(x',x_1))}{x_1^{k-\gv}} \bigg)^{p+q} \frac{dx_1}{x_1}.
\eal
Now, for $0<r<R/2$ we see that
\bal
\int_0^R \bigg( \frac{\nu(B^k(x',x_1))}{x_1^{k-\gv}} \bigg)^{p+q} \frac{dx_1}{x_1} \geq \int_r^{2r} \bigg( \frac{\nu(B^k(x',x_1))}{x_1^{k-\gv}} \bigg)^{p+q} \frac{dx_1}{x_1} \gtrsim \bigg( \frac{\nu(B^k(0,r))}{r^{k-\gv}} \bigg)^{p+q},
\eal
and taking the supremum over all such $r$ yields
\bal
\int_0^R \bigg( \frac{\nu(B^k(x',x_1))}{x_1^{k-\gv}} \bigg)^{p+q} \frac{dx_1}{x_1} \gtrsim \sup_{r\in(0,R/2)} \bigg( \frac{\nu(B^k(0,r))}{r^{k-\gv}} \bigg)^{p+q}.
\eal
Let us define the maximal function
\bal
M_{\gv,R/2}(x'):= \sup_{r\in(0,R/2)} \frac{\nu(B^k(x',r))}{r^{k-\gv}}.
\eal
Then, since $\nu$ has compact support in $B^k(0,R/2)$,
\bal
&\hspace{-1cm} \int_0^R x_1^{N-k-1-\xb} \bigg( \int_{B^k(0,R)} (x_1+|x'-y'|)^{-(N-2\xa_-)} d\nu(y') \bigg)^{p+q}dx_1 \\
&\gtrsim \int_{B^k(0,R)} M_{\gv,R/2}(x')^{p+q}\,dx' = \int_{\R^k} M_{\gv,R/2}(x')^{p+q}\,dx'.
\eal
The proof is again finished by an application of \cite[Lemma 2.3]{BVN}.
\end{proof}

\begin{proof}[\textbf{Proof of Theorem \ref{intro-Bessel-Omega}}]
Proceeding as in the proof of Theorem \ref{intro-Bessel-Sigma} and selecting $\Gamma=\partial\Omega$, it suffices to show that
\bal
\mathrm{Cap}^{(p+q)'}_{\CN_{2\xa_-,1},d^{p+1}d_\Sigma^{-(p+q+1)\xa_-}}(E\cap K_i) \approx \mathrm{Cap}^{(p+q)'}_{\CB_{N-1},\frac{2-q}{p+q}}(\tilde{T}_i(E\cap K_i))
\eal
for $i=1,\ldots,m$. Let $\xl\in\GTM_+(\partial\Omega)\cap\GTM_c(\partial\Omega\setminus\Sigma)$ be such that $\BBN_{2\xa_-,1}[\lambda] \in L^{p+q}(\Omega,d^bd_\Sigma^\theta)$ and set $d\lambda_{K_i}:= \chi_{K_i}d\lambda$. Then there is $\beta=\beta(\nu)>0$ such that $\Sigma_\beta\cap\supp\nu = \varnothing$. It is easy to show that 
\bal
&\int_\Omega \BBN_{2\xa_-,1}[\lambda_{K_i}]^{p+q}d^{p+1} d_\Sigma^{-(p+q+1)\xa_-}\,dx \\
&\hspace{1cm} \approx \sum_{i=1}^m \int_{O_i} \BBN_{2\xa_-,1}[\lambda_{K_i}]^{p+q}d^{p+1} d_\Sigma^{-(p+q+1)\xa_-}\,dx \\
&\hspace{2cm} \approx \sum_{i=1}^m \int_{O_i\setminus\Sigma_\beta} \BBN_{2\xa_-,1}[\lambda_{K_i}]^{p+q}d^{p+1}\,dx,
\eal
and note that for $x\in O_i\setminus\Sigma_\beta$ and $\xi\in\supp\nu$ we have that $d_\Sigma(x) \approx 1 \approx d_\Sigma(\xi)$ and $\CN_{2\xa_-,1}(x,\xi) \approx |x-\xi|^{-N}$.

Finally, we may apply \cite[Proposition 2.9]{BHV} (with parameters $\xa=\xb=0$, $s=(p+q)'$, $\xa_0=p+1$) to obtain the conclusion. Note that the extra conditions on $p,q$ are such that the proposition is applicable.
\end{proof}

\subsection{The case $\Sigma=\{0\}$}

Now we set some notation to treat the special case $\Sigma=\{ 0 \}$ and $\mu=H^2$ (in which $k=0$ and $\xa_-=N/2$). Let $0<\xe<N$, and set
\bal
\CN_\xe(x,y):= \frac{\max\{ |x-y|,|x|,|y|\}^N}{|x-y|^{N-2}\max\{ |x-y|,d(x),d(y)\}^2} + \max\{ |x-y|,d(x),d(y) \}^{-\xe}, \text{ and}
\eal
\bal
\CN_{N-\xe ,2}(x,y) := \frac{\max\{ |x-y|,|x|,|y|\}^{N-\xe}}{|x-y|^{N-2}\max\{ |x-y|,d(x),d(y)\}^2} \text{ as before.}
\eal
Moreover, set
\bal
G_{H^2,\xe}(x,y)&:= \frac{1}{|x-y|^{N-2}} \bigg( 1 \wedge \frac{d(x)d(y)}{|x-y|^2} \bigg) \bigg(1\wedge \frac{|x||y|}{|x-y|^2} \bigg)^{-N/2} \\
&\hspace{1cm} + \frac{d(x)d(y)}{(|x||y|)^{N/2}} \max\{ |x-y|,d(x),d(y) \}^{-\xe}, \text{ and}
\eal
\bal
\tilde{G}_{H^2,\xe}(x,y) := \frac{d(x)d(y)}{(|x||y|)^{N/2}} \CN_{N-\xe,2}(x,y)
\eal
Now recall estimate \eqref{Greenestb} for the Green kernel $G_{H^2}(x,y)$ and observe that there is a constant $C(\Omega,\xe)>0$ such that
\bal
|\ln (\min \{ |x-y|^{-2},d(x)^{-1}d(y)^{-1} \}) | \leq C(\Omega,\xe) \max\{ |x-y|,d(x),d(y) \}^{-\xe}.
\eal
This implies that
\bal
G_{H^2}(x,y) \lesssim G_{H^2,\xe}(x,y).
\eal
But it is also true that
\bal
G_{H^2,\xe}(x,y) \approx \frac{d(x)d(y)}{(|x||y|)^{N/2}} \CN_\xe(x,y)\quad \text{and}\quad \CN_\xe(x,y) \lesssim \CN_{N-\xe,2}(x,y),
\eal
hence
\bal
G_{H^2,\xe}(x,y) \lesssim \tilde{G}_{H^2,\xe}(x,y).
\eal
Finally set
\bal
\tilde{\BBG}_{H^2,\xe}[\tau](x):= \int_{\Omega} \tilde{G}_{H^2,\xe}(x,y)\,d\tau(y).
\eal

\begin{theorem}\label{0-existence-power-capacity}
Suppose that $\Sigma=\{ 0 \}$ and $\mu=H^2$. Let $p,q\geq 0$ with $p+q>1$ and let $\nu\in\GTM(\partial\Omega)$ be such that $\| \nu \|_{\GTM(\partial\Omega)}=1$, and assume that $$N+2-(N-2)p-Nq-2\xe>0.$$ Moreover, assume that
\bal
\nu(B) \leq C\,\mathrm{Cap}_{\BBN_{N-\xe,1},d^{p+1}d_0^{-\frac{N}{2}(p+q+1)}}^{(p+q)'}(B) \quad \forall B\in \CB(\overline{\Omega}).
\eal
Then the boundary value problem \eqref{power-BVP} possesses a non-negative weak solution provided that $\vgr$ is small enough.
\end{theorem}

\begin{proof}
Again we proceed as in the proof of Theorem \ref{Existence-capacity}. We apply Lemma \ref{Concrete-conditions}, now with parameters $\xa=N-\xe$, $\sigma=1$, $b=p+1$, $\theta= -\frac{N}{2}(p+q+1)$. We set the operator
\bal
\BBA[u]:=\BBG_{H^2}[|u|^p|\nabla u|^q]+\BBK_{H^2}[\vgr\nu],\quad u\in V
\eal
where
\bal
V:=\{ v\in W^{1,1}_{\loc}(\Omega): v\in L^{p+q}(\Omega,d^{1-q}d_0^{-N/2}) \text{ and } \nabla v \in L^{p+q}(\Omega,d^{p+1}d_0^{-N/2}) \}
\eal
It is a matter of straightforward calculations to show that
\bal
G_{H^2}(x,y) &\lesssim d(x)|x|^{-N/2}d(y)|y|^{-N/2}\CN_{N-\xe,1}(x,y), \\
|\nabla_x G_{H^2}(x,y)| &\lesssim |x|^{-N/2}d(y)|y|^{-N/2}\CN_{N-\xe,1}(x,y), \\
K_{H^2}(x,\xi) &\lesssim d(x)|x|^{-N/2}\CN_{N-\xe,1}(x,\xi), \\
|\nabla_x K_{H^2}(x,\xi)| &\lesssim |x|^{-N/2}\CN_{N-\xe,1}(x,\xi),
\eal
and also
\bal
|\BBA[u]|&\leq C_1 (dd_0^{-N/2}\BBN_{N-\xe,1}[dd_0^{-N/2}|u|^p|\nabla u|^q]+dd_0^{-N/2}\BBN_{N-\xe,1}[\vgr\nu]) \\
|\nabla \BBA[u]|&\leq C_1 (d_0^{-N/2}\BBN_{N-\xe,1}[dd_0^{-N/2}|u|^p|\nabla u|^q]+d_0^{-N/2}\BBN_{N-\xe,1}[\vgr\nu]).
\eal
Now set
\bal
E:=\{u\in W^{1,1}_{\loc}(\Omega): |u|\leq 2C_1 dd_0^{-N/2}\BBN_{N-\xe,1}[\vgr\nu],\ |\nabla u|\leq 2C_1 d_0^{-N/2}\BBN_{N-\xe,1}[\vgr\nu] \}
\eal
Using \eqref{concrete-condition-2} we see that $E\subset V$. Moreover, due to \eqref{concrete-condition-3} there is $\vgr_0$ such that
\bal
\vgr\in(0,\vgr_0) \Rightarrow \BBA[E]\subset E.
\eal
$E$ is easily seen to be closed and convex in $V$, and one can show that $\BBA:E\rightarrow \BBA[E]$ is continuous and compact. Thus the Schauder fixed point theorem applies, yielding a solution $u\in E$ of $u=\BBA[u]$, which is consequently a weak solution of \eqref{power-BVP}. Non-negativity can be obtained by restriction to the positive cone of $E$.
\end{proof}

\begin{proof}[\textbf{Proof of Theorem \ref{intro-0-bvp}}]
In view of Lemma \ref{Concrete-conditions} and the proof of Theorem \ref{0-existence-power-capacity}, it suffices to show that
\bal
\int_B \BBN_{N-\xe}[\chi_B \nu]^{p+q}d^{p+1}(x)|x|^{-\frac{N}{2}(p+q+1)}\,dx \leq C \nu(B) \quad \forall B\in \CB(\overline{\Omega}).
\eal

For (1), suppose that $\nu=\delta_0$. If $0\notin B$ then $\chi_B \delta_0 = 0$ and we have nothing to show, so assume that $0\in B$. Then $\chi_B \delta_0= \delta_0$ and 
\bal
&\int_B \BBN_{N-\xe}[\delta_0]^{p+q}d^{p+1}(x)|x|^{-\frac{N}{2}(p+q+1)}\,dx \\
&\hspace{1cm}= \int_B d^{p+1}|x|^{-\frac{N}{2}(p+q+1)-\xe(p+q)}\,dx \\
&\hspace{2cm}= \int_B |x|^{-\frac{N}{2}(p+q+1)-\xe(p+q)+p+1}\,dx \\
&\hspace{3cm}\approx \int_0^{\operatorname{diam}(\Omega)} r^{-\frac{N}{2}(p+q+1)-\xe(p+q)+p+N}\,dr
\eal
and observe that the last integral converges provided that $-\frac{N}{2}(p+q+1)-\xe(p+q)+p+N>-1$.

For (2), let $\supp \nu \subset \partial \Omega \setminus \{ 0\}$. Now observe that 
\bal
\CN_{N-\xe}(x,\xi) \approx d^{-1}(x)|x|^{N/2} K_{\mu}(x,\xi)
\eal
where $\mu= \frac{N^2-\xe^2}{4}< \frac{N^2}{4}$ (so that estimate \eqref{Martinest1} applies). It follows that 
\bal
&\int_B \BBN_{N-\xe}[\chi_B\nu]^{p+q}d^{p+1}(x)|x|^{-\frac{N}{2}(p+q+1)}\,dx \\
&\hspace{1cm} \approx \int_B \BBK_\mu[\chi_B\nu]^{p+q} d^{1-q}(x)|x|^{-N/2}\,dx \\
&\hspace{2cm} = \int_B (\BBK_\mu[\chi_B\nu]d^{-\frac{q}{p+q}}(x))^{p+q} d(x)|x|^{-N/2}\,dx \\
&\hspace{3cm} \lesssim \| \BBK_{\mu,\xg}[\nu] \|_{L^{p+q}(\Omega;\ei)} \leq C \nu(B),
\eal
provided that $p+q< \frac{N+1}{N+\xg-1}$ (so that Theorem \ref{martin-est-p1-p2} applies).
\end{proof}

\begin{theorem} \label{0-boundary-existence}
Suppose that $\Sigma=\{0\}$ and $\mu=H^2$. Let $p,q\geq 0$ with $p+q>1$ and $N+2-(N-2)p-Nq-2\xe>0$, and let $\nu\in \GTM_+(\partial \Omega)$ with $\supp \nu \subset \partial\Omega\setminus\{0\}$ be such that
\bal
\nu(B) \leq C \mathrm{Cap}^{\partial\Omega,(p+q)'}_{\frac{2-q}{p+q}}(B) \quad \forall B\in \CB(\partial\Omega),
\eal
where we assume that $q<2$ and $N(p+q)>p+2$. Then \eqref{power-BVP} possesses a non-negative weak solution provided that $\vgr\in (0,\vgr_0)$ for small enough $\vgr_0$.
\end{theorem}

\begin{proof}
We proceed as in the proof of Theorem \ref{intro-Bessel-Omega}. Selecting $\Gamma=\partial\Omega$, in view of Theorem \ref{0-existence-power-capacity} it suffices to show that
\bal
\mathrm{Cap}^{(p+q)'}_{\CN_{N-\xe,1},d^{p+1}d_0^{-\frac{N}{2}(p+q+1)}}(E\cap K_i) \approx \mathrm{Cap}^{(p+q)'}_{\CB_{N-1},\frac{2-q}{p+q}}(\tilde{T}_i(E\cap K_i))
\eal
for $i=1,\ldots,m$. Let $\xl\in\GTM_+(\partial\Omega)\cap\GTM_c(\partial\Omega\setminus\Sigma)$ be such that 
\bal \BBN_{N-\xe,1}[\lambda] \in L^{p+q}(\Omega,d^{p+1}d_0^{-\frac{N}{2}(p+q+1)})\eal 
and set $d\lambda_{K_i}:= \chi_{K_i}d\lambda$. Then there is $\beta=\beta(\nu)>0$ such that $B(0,\beta)\cap\supp\nu = \varnothing$. It is easy to show that 
\bal
&\int_\Omega \BBN_{N-\xe,1}[\lambda_{K_i}]^{p+q}d^{p+1} |x|^{-\frac{N}{2}(p+q+1)}\,dx \\
&\hspace{1cm} \approx \sum_{i=1}^m \int_{O_i} \BBN_{N-\xe,1}[\lambda_{K_i}]^{p+q}d^{p+1} |x|^{-\frac{N}{2}(p+q+1)}\,dx \\
&\hspace{2cm} \approx \sum_{i=1}^m \int_{O_i\setminus B(0,\beta)} \BBN_{N-\xe,1}[\lambda_{K_i}]^{p+q}d^{p+1}\,dx,
\eal
and note that for $x\in O_i\setminus B(0,\beta)$ and $\xi\in\supp\nu$ we have that $|x| \approx 1 \approx |\xi|$ and $\CN_{N-\xe,1}(x,\xi) \approx |x-\xi|^{-N}$.

Finally, we may apply \cite[Proposition 2.9]{BHV} (with parameters $\xa=\xb=0$, $s=(p+q)'$, $\xa_0=p+1$) to obtain the conclusion. Note that the extra conditions on $p,q$ are such that the proposition is applicable.
\end{proof}


\end{document}